\documentclass[10pt]{amsart}

\usepackage{macros}
\standardsettings

\newcommand{\xqedhere}[2]{\rlap{\hbox to#1{\hfil\llap{\ensuremath{#2}}}}}

\newcommand{\firstgreek}{\mu}
\newcommand{\secondgreek}{\nu}
\newcommand{\pgreek}{\firstgreek}
\newcommand{\qgreek}{\secondgreek}
\newcommand{\greeks}{{\firstgreek,\secondgreek}}
\newcommand{\Vgreeks}{V_\firstgreek V_\secondgreek}

\newcommand{\firstdim}{m}
\newcommand{\seconddim}{n}
\newcommand{\pdim}{\firstdim}
\newcommand{\qdim}{\seconddim}
\newcommand{\dims}{{\firstdim,\seconddim}}
\newcommand{\dimsum}{{\firstdim + \seconddim}}
\newcommand{\dimprod}{{\firstdim \seconddim}}
\newcommand{\Dimsum}{d}
\newcommand{\Dimprod}{D}
\newcommand{\generic}{d}

\newcommand{\constants}{{\dims,\greeks}}

\newcommand{\dnorm}{\infty}
\newcommand{\pnorm}{\dnorm}
\newcommand{\qnorm}{\dnorm}
\newcommand{\norms}{{\pnorm,\qnorm}}
\newcommand{\onenorm}{1}

\newcommand{\diamcube}{\|\KK\|}
\renewcommand{\Codiam}[1]{\|#1\|}
\newcommand{\cdl}{{\Codiam{\Lambda}}}
\newcommand{\cdd}{{\|\DD\|}}

\DeclareMathOperator{\Irr}{Irr}
\DeclareMathOperator{\boxc}{N}
\newcommand{\ballsymbol}{B}
\newcommand{\ball}[2]{\ballsymbol_{#2}(#1)}
\newcommand{\ballvar}[2]{\ballsymbol(#1,#2)}
\newcommand{\multerror}{\gamma}
\newcommand{\bigast}{\mathop{\text{\scalebox{2}{$\ast$}}}}
\newcommand{\Approx}[1]{W(#1)}
\newcommand{\Approxmod}[2]{W_{#2}(#1)}
\newcommand{\Bad}[1]{B(#1)}
\newcommand{\Badmod}[2]{B_{#2}(#1)}

\newcommand{\Wedge}{\wedge}
\newcommand{\eucq}{\approx}

\newcommand{\minkowski}{\lambda_1}

\newcommand{\Lrep}{\Lambda_\prim^+}

\DeclareMathOperator{\lebesgue}{\lambda}
\DeclareMathOperator{\haarg}{\lambda_\GG}
\DeclareMathOperator{\haar}{\lambda_\Omega}

\theoremstyle{definition}
\newtheorem{quest}{Question}
\newtheorem{heur}{Heuristic assumption}

\renewcommand{\eqsim}[1]{\underset{#1}{\sim}}
\newcommand{\leqsim}[1]{\underset{#1}{\lesssim}}
\newcommand{\geqsim}[1]{\underset{#1}{\gtrsim}}
\newcommand{\LessLess}[1]{\underset{#1}{\lessless}}
\newcommand{\Plussim}[1]{\underset{#1}{\sim_\plus}}

\draftfalse

\begin{document}
\title[A Hausdorff measure version of the Jarn\'ik--Schmidt theorem]{A Hausdorff measure version of the Jarn\'ik--Schmidt theorem in Diophantine approximation}

\authordavid


\begin{Abstract}
We solve the problem of giving sharp asymptotic bounds on the Hausdorff dimensions of certain sets of badly approximable matrices, thus improving results of Broderick and Kleinbock (preprint 2013) as well as Weil (preprint 2013), and generalizing to higher dimensions those of Kurzweil ('51) and Hensley ('92). In addition we use our technique to compute the Hausdorff $f$-measure of the set of matrices which are not $\psi$-approximable, given a dimension function $f$ and a function $\psi:(0,\infty)\to (0,\infty)$. This complements earlier work by Dickinson and Velani ('97) who found the Hausdorff $f$-measure of the set of matrices which are $\psi$-approximable.
\end{Abstract}

\clearpage
\maketitle

\ignore{
\draftnewpage
\section{Introduction - fast paced version}

\subsection{Superlevelsets of the Lagrange constant function}
Fix $\dims\in\N$, and let $\MM$ denote the set of $\pdim\times \qdim$ matrices. A matrix $A\in\MM$ is said to be \emph{badly approximable} if there exists $\kappa > 0$ such that
\begin{equation}
\label{kappadef}
\|A\qq - \pp\|^\pdim \|\qq\|^\qdim \geq \kappa \text{ for all }\qq\in\Z^\qdim\butnot\{\0\} \text{ and } \pp\in\Z^\pdim.
\end{equation}
The supremum of the set of $\kappa$ satisfying \eqref{kappadef} is called the \emph{Lagrange constant}\Footnote{The phrase ``Lagrange constant'' is also sometimes uzed to refer to the reciprocal of this number.} of $A$, and we denote it by $\LL(A)$. Note that the Lagrange constant of a matrix depends on the choice of norms used in the formula \eqref{kappadef}; if we use the norms $\|\cdot\|_\pgreek$ and $\|\cdot\|_\qgreek$ on $\R^\firstdim$ and $\R^\seconddim$, respectively, then we denote the corresponding Lagrange constant by $\LL_{\greeks}(A)$. Throughout this paper, $\|\cdot\|_\dnorm$ will denote the max norm on $\R^\generic$. In particular, $\|\cdot\|_\onenorm$ is just another notation for the absolute value function on $\R$.

Dirichlet's theorem in Diophantine approximation implies that $\LL_\norms(A) \leq 1$ for any matrix $A\in\MM$. It is complemented by the Jarn\'ik--Schmidt theorem, which says that the Hausdorff dimension of the badly approximable $\pdim\times\qdim$ matrices (i.e. those with positive Lagrange constant) is equal to $\dimprod$, the dimension of the space of all $\pdim\times\qdim$ matrices. In this paper, we consider the Hausdorff dimensions of the sets
\[
\Badmod\kappa\greeks \df \{A\in\MM : \LL_{\greeks}(A) \geq \kappa\} \;\;\; (\kappa > 0),
\]
i.e. the superlevelsets of the function sending a matrix to its Lagrange constant. Our main result concerning these sets is the following:
}

\draftnewpage
\section{Introduction}
\label{sectionintroduction}

\subsection{Notation}

Fix $\dims\in\N$, and let $\MM$ denote the set of $\pdim\times \qdim$ matrices. Given a function $\psi:(0,\infty)\to(0,\infty)$ and two norms $\|\cdot\|_\firstgreek$, $\|\cdot\|_\secondgreek$ on $\R^\firstdim$ and $\R^\seconddim$ respectively, a matrix $A\in \MM$ is called \emph{$\psi$-approximable} (with respect to the norms $\|\cdot\|_\firstgreek$ and $\|\cdot\|_\secondgreek$) if there exist infinitely many vectors $\pp\in\Z^\pdim$ and $\qq\in\Z^\qdim\butnot\{\0\}$ such that
\[
\|A\qq - \pp\|_\pgreek \leq \psi(\|\qq\|_\qgreek).
\]
If $\|\cdot\|_\qgreek$ is the max norm, then we may allow $\psi:\N\to (0,\infty)$ in place of $\psi:(0,\infty)\to (0,\infty)$. We denote the set of $\psi$-approximable matrices by $\Approx\psi$, or by $\Approxmod\psi\greeks$ when we wish to emphasize the role of the norms $\|\cdot\|_\firstgreek$ and $\|\cdot\|_\secondgreek$.

\begin{remark*}
When $\qdim = 1$, then $\Approx\psi$ is the set of simultaneously $\psi$-approximable points in $\R^\pdim$, and when $\pdim = 1$, then $\Approx\psi$ is the set of dually $\psi$-approximable points in $\R^\qdim$. This convention agrees with most sources, but some swap the roles of $\firstdim$ and $\seconddim$ (e.g. \cite{BoveyDodson, DickinsonVelani, Weil1}).
\end{remark*}


\subsection{A brief history of Diophantine approximation}
Diophantine approximation is traditionally considered to have begun in 1842, when Dirichlet proved the following theorem:

\begin{theorem*}[Dirichlet's Theorem]
For any $A\in \MM$ and for any $Q > 1$, there exist $\pp\in\Z^\pdim$ and $\qq\in\Z^\qdim\butnot\{\0\}$ such that
\[
\|\qq\|_\qnorm \leq Q \text{ and } \|A\qq - \pp\|_\pnorm \leq Q^{-\qdim/\pdim},
\]
where $\|\cdot\|_\dnorm$ denotes the max norm. In particular, if
\[
\psi_\ast(q) = q^{-\qdim/\pdim}
\]
then
\begin{equation}
\label{corollarydirichlet}
\Approxmod{\psi_\ast}\norms = \MM.
\end{equation}
\end{theorem*}

We will refer to the equation \eqref{corollarydirichlet} as \emph{the corollary to Dirichlet's theorem}.

A natural question is whether Dirichlet's theorem or its corollary can be improved. The first result in this direction is due to Liouville, who showed in 1844 that the set of \emph{badly approximable matrices}
\[
\BA \df \bigcup_{\multerror > 0}\MM\butnot \Approx{\multerror\psi_\ast}
\]
is nonempty when $\firstdim = \seconddim = 1$. The assumption $\firstdim = \seconddim = 1$ was removed by O. Perron in 1921 \cite{Perron}.

Liouville and Perron's result shows that for any function $\psi$ satisfying $\psi/\psi_\ast\to 0$, we have $\Approx\psi \neq\MM$ (cf. \cite[Section 2]{FishmanSimmons5} for an extended discussion of this point). Nevertheless, for such a $\psi$ the set $\Approx\psi$ may still be ``large'' in a number of senses. For example, it is trivial to show that for any function $\psi$, $\Approx\psi$ is comeager.\Footnote{We recall that a subset of a Baire space is called \emph{comeager} or \emph{residual} if it contains a dense $G_\delta$ set. By the Baire category theory, the class of comeager sets is closed under countable intersections. The complement of a comeager set is called \emph{meager} or \emph{of the first category}.} The first nontrivial result concerning the size of $\Approx\psi$ came in 1924, when A. Y. Khinchin proved the case $\firstdim = \seconddim = 1$ of the theorem below \cite{Khinchin1}. Two years later Khinchin extended his result to include the case $\qdim = 1$, $\pdim$ arbitrary \cite{Khinchin2}; the general case was proven by A. V. Groshev in 1938 \cite{Groshev}.

\begin{theorem*}[The Khinchin--Groshev theorem]
Fix $\psi:\N\to (0,\infty)$, and consider the series
\begin{equation}
\label{khinchinseries}
\sum_{q = 1}^\infty q^{\qdim - 1}\psi^\pdim(q).
\end{equation}
\begin{itemize}
\item[(i)] If \eqref{khinchinseries} converges, then $\Approxmod\psi\norms$ is of Lebesgue measure zero.
\item[(ii)] If \eqref{khinchinseries} diverges and the function $q\mapsto q^\qdim\psi^\pdim(q)$ is nonincreasing,\Footnote{Weakening this assumption has been the motivation for much further research in this area, culminating in its complete elimination in the case $(\pdim,\qdim)\neq (1,1)$ \cite{BeresnevichVelani2}. When $\pdim = \qdim = 1$, the theorem is false without the assumption, and the expected measure of $\Approx\psi$ for arbitrary $\psi$ is described by the Catlin conjecture \cite[p.28]{Harman}.} then $\Approxmod\psi\norms$ is of full Lebesgue measure (i.e. its complement has measure zero).
\end{itemize}
\end{theorem*}

\begin{remark*}
Case (i) (the ``convergence case'') of the Khinchin--Groshev theorem is a standard application of the Borel--Cantelli lemma; case (ii) (the ``divergence case'') is the difficult direction.
\end{remark*}

Fix $c > 0$, and let $\psi_c(q) = q^{-c}$. One naturally observes that \eqref{khinchinseries} converges if $c > \qdim/\pdim$, and thus in this case the sets $\Approx{\psi_c}$ are of Lebesgue measure zero. Nevertheless, the sets $\Approx{\psi_c}$ still turn out to have positive Hausdorff dimension. The following result was proven for $\qdim = 1$, $\pdim$ arbitary by V. Jarn\'ik in 1929 \cite{Jarnik3}, and in general by J. D. Bovey and M. M. Dodson in 1986 \cite{BoveyDodson}. Moreover, the case $\firstdim = \seconddim = 1$ was proven independently by A. S. Besicovitch in 1934 \cite{Besicovitch}.

\begin{theorem}[The Jarn\'ik--Besicovitch--Bovey--Dodson Theorem]
\label{theoremJBBD}
For $c > \qdim/\pdim$,
\[
\HD(\Approx{\psi_c}) = (\qdim - 1)\pdim + \frac{\dimsum}{1 + c}\cdot
\]
\end{theorem}
Here and throughout, $\HD(S)$ denotes the Hausdorff dimension of a set $S$.

\begin{remark*}
In Theorem \ref{theoremJBBD}, the limiting dimension $\lim_{c\to\infty} \HD(\Approx{\psi_c}) = (\qdim - 1)\pdim$ is not zero unless $\qdim = 1$. The reason for this is that for every function $\psi$, $\Approx\psi$ contains the set $\{A\in\MM: A\qq \in \Z^\pdim\text{ for some }\qq\in \Z^\qdim\butnot\{\0\}\}$, and it can be computed that the dimension of this set is $(\qdim - 1)\pdim$. Thus by monotonicity of Hausdorff dimension, we have $\HD(\Approx\psi) \geq (\qdim - 1)\pdim$ for every function $\psi$.
\end{remark*}

In fact, Jarn\'ik proved a stronger statement than the above. Given a dimension function $f$,\Footnote{We recall that a \emph{dimension function} is a nondecreasing function $f:(0,\infty)\to(0,\infty)$ such that $\lim_{\rho\to 0} f(\rho) = 0$.} let $\HH^f$ denote the Hausdorff $f$-measure (see Section \ref{sectionHD}). Let $f_\ast:(0,\infty)\to(0,\infty)$ be the function defined by the formula
\begin{equation}
\label{fastdef}
f_\ast(\rho) = \rho^\dimprod.
\end{equation}
The following result was proven for $\qdim = 1$, $\pdim$ arbitrary by Jarn\'ik in 1929 \cite{Jarnik3}, and in general by D. Dickinson and S. L. Velani in 1997 \cite{DickinsonVelani}:

\begin{theorem}[The Jarn\'ik--Dickinson--Velani Theorem]
\label{theoremJDV}
Let $\psi:\N\to (0,\infty)$, let $f$ be a dimension function, and assume that
\begin{itemize}
\item the functions $\psi/\psi_\ast$ and $f/f_\ast$ are nonincreasing and satisfy
\begin{align*}
\frac{\psi(q)}{\psi_\ast(q)} &\tendsto{q\to\infty} 0,&
\frac{f(\rho)}{f_\ast(\rho)} &\tendsto{\rho \to 0} \infty;
\end{align*}
\item the function
\[
q\mapsto \left(\frac{\psi(q)}{\psi_\ast(q)}\right)^\pdim \frac{f(\Psi(q))}{f_\ast(\Psi(q))}
\]
is nonincreasing, where
\[
\Psi(q) = \psi(q)/q.
\]
\end{itemize}
Consider the series
\begin{equation}
\label{JDVseries}
\sum_{q = 1}^\infty \frac{1}{q} \left(\frac{\psi(q)}{\psi_\ast(q)}\right)^\pdim \frac{f(\Psi(q))}{f_\ast(\Psi(q))}\cdot
\end{equation}
Then
\[
\HH^f(\Approxmod\psi\norms) = \begin{cases}
0 & \text{if \eqref{JDVseries} converges}\\
\infty & \text{if \eqref{JDVseries} diverges}
\end{cases}.
\]
\end{theorem}

Note that Theorem \ref{theoremJDV} includes Theorem \ref{theoremJBBD} as a corollary. Indeed, setting $f(\rho) = \rho^s$ and $\psi = \psi_c$ yields Theorem \ref{theoremJBBD}.

\begin{remark*}
As with the Khinchin--Groshev theorem, the convergence case of Theorem \ref{theoremJDV} is a straightforward application of the Hausdorff--Cantelli lemma \cite[Lemma 3.10]{BernikDodson}, while the divergence case is the difficult direction.
\end{remark*}

One may also ask about the characteristics of the sets $\Bad\psi \df \MM\butnot \Approx\psi$. From the measure-theoretic and topological points of view, there is no new information here; $\Bad\psi$ is meager, and it is of measure zero if and only if $\Approx\psi$ is of full measure. However, from the dimension point of view there is more to ask. In this context, the main result was proven by Jarn\'ik in 1928 for the case $\firstdim = \seconddim = 1$. The general case was proven by W. M. Schmidt in 1969 \cite{Schmidt2}, using the game he had introduced three years earlier \cite{Schmidt1} which is now known as Schmidt's game.

\begin{theorem*}[The Jarn\'ik--Schmidt Theorem]
The set $\BA = \bigcup_{\multerror > 0}\Bad{\multerror\psi_\ast}$ has Hausdorff dimension $\dimprod$.
\end{theorem*}

Since $\dimprod$ is the largest possible dimension, this closes the question of the dimension of $\Bad\psi$ for functions $\psi$ such that $\psi/\psi_\ast\to 0$, since such functions satisfy $\BA \subset \Bad\psi$. On the other hand, for functions $\psi$ such that $\psi\geq \psi_\ast$, we have $\Badmod\psi\norms = \emptyset$ by the corollary to Dirichlet's theorem. To summarize,
\[
\HD(\Badmod\psi\norms) = \begin{cases}
mn & \text{if } \psi/\psi_\ast \to 0\\
0 & \text{if } \psi \geq \psi_\ast\\
\text{unknown} & \text{otherwise}
\end{cases}.
\]

\subsection{Main results}

Although the Jarn\'ik--Schmidt theorem tells us the Hausdorff dimension of $\Bad\psi$ for any function $\psi$ satisfying $\psi/\psi_\ast\to 0$, it leaves open the following two natural questions:

\begin{quest}
\label{question1}
What are the dimensions of the sets $\Bad\kappa$ ($\kappa > 0$), where for each $\kappa > 0$
\[
\Bad\kappa = \Badmod\kappa\greeks \df \Badmod{\kappa^{1/\pdim}\psi_\ast}\greeks
= \left\{A\in\MM : \begin{split}\|A\qq - \pp\|_\pgreek^\pdim \|\qq\|_\qgreek^\qdim > \kappa \text{ for all but}\\ \text{finitely many $(\pp,\qq)\in \Z^\dimsum$}\end{split}\right\} \;\;\;\;?
\]
\end{quest}
\begin{quest}
\label{question2}
Given a dimension function $f$ and a function $\psi$ such that $\psi/\psi_\ast\to 0$, what is $\HH^f(\Bad\psi)$?
\end{quest}

In the present paper, we give fairly complete answers to both of these questions. To Question \ref{question1}, we give the following asymptotic answer:

\begin{theorem}
\label{theoremHDasymp}
Given any two norms $\|\cdot\|_\pgreek$ and $\|\cdot\|_\qgreek$ on $\R^\firstdim$ and $\R^\qdim$ respectively, we have
\[
\lim_{\kappa\to 0} \frac{\dimprod - \HD(\Badmod\kappa\greeks)}{\kappa} = \theta_{\greeks} \df \frac{\Vgreeks}{2\zeta(\dimsum)}\frac{\dimprod}{\dimsum}
\]
or equivalently
\begin{equation}
\label{thetamn}
\HD(\Badmod\kappa\greeks) = \dimprod - \theta_{\greeks}\kappa + o(\kappa).
\end{equation}
Here $V_\firstgreek$ (resp. $V_\secondgreek$) denotes the volume of the unit ball in $\R^\firstdim$ (resp. $\R^\seconddim$) with respect to the $\|\cdot\|_\firstgreek$ (resp. $\|\cdot\|_\secondgreek$) norm, $\zeta$ denotes the Riemann zeta function, and $\HD$ denotes Hausdorff dimension.
\end{theorem}

This theorem has been preceded by many partial results. When $\firstdim = \seconddim = 1$, J. Kurzweil \cite[Theorem VII]{Kurzweil} proved that
\begin{equation}
\label{kurzweil}
1 - .99\kappa \leq \HD(\Badmod\kappa{\onenorm,\onenorm}) \leq 1 - .25 \kappa
\end{equation}
for all sufficiently small $\kappa$. Here $\|\cdot\|_\onenorm$ denotes the absolute value function on $\R$. Later D. Hensley \cite{Hensley} (cf. \cite[Theorem 1.9]{Bugeaud}) improved this result by showing that \hspace{2 in}
\begin{equation}
\label{hensley}
\HD(\Badmod\kappa{\onenorm,\onenorm}) = 1 - \frac{6}{\pi^2}\kappa - \frac{72}{\pi^4} \kappa^2 |\log(\kappa)| + O(\kappa^2).
\end{equation}
Of course, \eqref{hensley} is stronger than \eqref{kurzweil} since $.25 < 6/\pi^2 < .99$. \eqref{hensley} is also slightly stronger than \eqref{thetamn}, since it has an estimate on the error term $o(\kappa)$. Here is the calculation which checks that \eqref{thetamn} and \eqref{hensley} agree on the second term:
\[
\theta_{\onenorm,\onenorm} = \frac{2\cdot 2}{2\zeta(2)}\frac{1\cdot 1}{1 + 1} = \frac{1}{\zeta(2)} = \frac{6}{\pi^2}\cdot
\]
When $\qdim = 1$ and $\pdim$ is arbitrary, S. Weil \cite{Weil1} showed that
\[
\pdim - k_1 \frac{\kappa^{1/(2\pdim)}}{|\log(\kappa)|} \leq \HD(\Bad\kappa) \leq \pdim - k_2 \frac{\kappa^{\pdim + 1}}{|\log(\kappa)|}
\]
for some constants $k_1,k_2 > 0$ depending on $\pdim$. This was improved by R. Broderick and D. Y. Kleinbock \cite{BroderickKleinbock} who showed that for $\dims$ both arbitrary,
\begin{equation}
\label{broderickkleinbock}
\dimprod - k_3 \frac{\kappa^{1/p(\dims)}}{|\log(\kappa)|} \leq \HD(\Bad\kappa) \leq \dimprod - k_4 \frac{\kappa}{|\log(\kappa)|}\cdot
\end{equation}
for some constants $k_3,k_4, p(\dims) > 0$ depending on $\firstdim$ and $\seconddim$, with the property that $p(\pdim,1) = 2\pdim \all \pdim$. Note that the results of Kurzweil, Weil, and Broderick--Kleinbock all follow immediately from \eqref{thetamn}, while the result of Hensley does not. The conjecture of Broderick and Kleinbock made on \cite[last para. of p.3]{BroderickKleinbock} also follows immediately from \eqref{thetamn}.

\begin{remark*}
Although the result \eqref{broderickkleinbock} is superceded by \eqref{thetamn}, the left-hand inequality of \eqref{broderickkleinbock} is still interesting in that it was proven using Schmidt games \cite{Schmidt1}, and appears to be the optimal result that one can prove using this technique. Thus the left-hand inequality of \eqref{broderickkleinbock} illustrates both the strengths and weaknesses of the Schmidt games technique when applied to problems more precise than simply asking for full dimension of $\BA$.
\end{remark*}

\begin{remark*}
Although Theorem \ref{theoremHDasymp} provides asymptotic information about the nonincreasing function $f_{\greeks}(\kappa) = \HD(\Badmod\kappa\greeks)$, the following questions concerning the behavior of $f_\greeks$ are still open:
\begin{itemize}
\item[1.] Is $f_{\greeks}$ continuous? It was recently proven by C. G. T. de A. Moreira \cite{Moreira} that $f_{\onenorm,\onenorm}$ is continuous, but his techniques depend heavily on continued fractions and do not generalize to higher dimensions, where the question remains open.
\item[2.] What is the supremum of $\kappa$ for which $f_{\greeks}(\kappa) > 0$? Again, the answer is known only in dimension one, where Moreira showed that $f_{\onenorm,\onenorm}(\kappa) = 0$ if and only if $\kappa \geq 1/3$ \cite[Theorem 1(ii)]{Moreira}.
\end{itemize}
\end{remark*}

\begin{remark*}
It is worthwhile to note that Theorem \ref{theoremHDasymp} provides an alternate proof of the Jarn\'ik--Schmidt Theorem. The existing proofs \cite{Schmidt2, Fishman2, BFS1} have all been minor variations of each other, so it is nice to get an independent proof. (It should be noted here that when $\firstdim = 1$ or $\seconddim = 1$, there are easier proofs available \cite{KTV, Fishman, BFKRW} which rely on the so-called Simplex Lemma \cite[Lemma 4]{KTV}.)
\end{remark*}


Regarding Question \ref{question2}, we prove the following theorem, which should be thought of as an analogue of Theorem \ref{theoremJDV} for the sets $\Bad\psi$ ($\psi:(0,\infty)\to (0,\infty)$):

\begin{theorem}
\label{maintheorem}
Let $f$ be a dimension function, and let $\psi:(0,\infty)\to (0,\infty)$ be a function such that \eqref{khinchinseries} diverges but $\psi/\psi_\ast$ is nonincreasing and tends to zero. Let
\begin{equation}
\label{Lfpsi}
L_{f,\psi} \df \liminf_{\rho\to 0} \frac{\log\left(\frac{f(\rho)}{f_\ast(\rho)}\right)}{F_\psi(\rho^{-\pdim/(\dimsum)})}
\end{equation}
(cf. \eqref{fastdef}), where
\begin{equation}
\label{Fpsidef1}
F_\psi(Q) = \sum_{q = 1}^Q q^{\qdim - 1}\psi^\pdim(q).
\end{equation}
Then
\begin{equation}
\label{HfBpsi}
\HH^f(\Badmod\psi\greeks)
= \begin{cases}
0 & \text{if } L_{f,\psi} < \eta_{\greeks}\\
\infty & \text{if } L_{f,\psi} > \eta_{\greeks}\\
\text{unknown} & \text{if } L_{f,\psi} = \eta_{\greeks}
\end{cases},
\end{equation}
where
\begin{equation}
\label{etadef}
\eta_{\greeks} = \frac{\qdim \Vgreeks}{2\zeta(\dimsum)} = \frac{\dimsum}{\pdim} \theta_{\greeks}\cdot
\end{equation}
\end{theorem}

\begin{remark*}
The divergence case of the Khinchin--Groshev theorem is a corollary of Theorem \ref{maintheorem}. Indeed, let $\psi:\N\to (0,\infty)$ be a function such that \eqref{khinchinseries} diverges but $\psi/\psi_\ast$ is nonincreasing and tends to zero. Since $\HH^{f_\ast}$ is Lebesgue measure and $L_{f_\ast,\psi} = 0$, Theorem \ref{maintheorem} says that the set $\Badmod\psi\norms$ has Lebesgue measure zero, or equivalently that the set $\Approxmod\psi\norms$ has full Lebesgue measure. If $\psi/\psi_\ast$ is nonincreasing but does not tend to zero, then $\Approxmod\psi\norms$ can be seen to have full measure via comparison with the function $\psi(q) = \psi_\ast(q)/\log(2\vee q)$.\Footnote{Here and hereafter $A\vee B$ denotes the maximum of two numbers $A$ and $B$, and $A\wedge B$ denotes their minimum.}
\end{remark*}

\begin{definition}
\label{definitionniceapproximation}
In what follows, we shall call a function $\psi:(0,\infty)\to (0,\infty)$ such that \eqref{khinchinseries} diverges but $\psi/\psi_\ast$ is nonincreasing a \emph{nice approximation function}.
\end{definition}


To illustrate Theorem \ref{maintheorem}, we consider its special case for a certain family of nice approximation functions $(\psi_\multerror)_{\multerror > 0}$, and we find a ``dual'' family of dimension functions $(f_s)_{s > 0}$ which are useful for measuring the sets $(\Approx{\psi_\multerror})_{\multerror > 0}$:

\begin{example}
\label{examplefs}
For each $\multerror,s > 0$, let
\begin{align*}
\psi_\multerror(q) &= \multerror^{1/\pdim} q^{-\qdim/\pdim}\log(2\vee q)^{-1/\pdim},&
f_s(\rho) &= \rho^{\dimprod}|\log(\rho)|^s,
\end{align*}
we have
\begin{align*}
F_\psi(Q) &= \multerror\sum_{q = 1}^Q \log(2\vee q)^{-1} \frac{1}{q} \sim \multerror\log\log(Q)\\
\log\left(\frac{f_s(\rho)}{f_\ast(\rho)}\right) &= \log\left(|\log(\rho)|^s\right) = s\log|\log(\rho)|\\
L_{f_s,\psi_\multerror} &= \liminf_{\rho \to 0} \frac{s\log|\log(\rho)|}{\multerror \log\log(\rho^{-\pdim/(\dimsum)})} = \frac{s}{\multerror}\cdot
\end{align*}
So by Theorem \ref{maintheorem}, we have
\[
\HH^{f_s}(\Bad{\psi_\multerror})
= \begin{cases}
0 & \text{if } s < \multerror\eta_{\greeks}\\
\infty & \text{if } s > \multerror\eta_{\greeks}\\
\text{unknown} & \text{if } s = \multerror\eta_{\greeks}
\end{cases}.
\]
\end{example}

Example \ref{examplefs} reveals a surprising fact about the sets $\Bad\psi$, namely that there exist pairs $(f,\psi)$ and constants $0 < \multerror_- < \multerror_+$ such that $\HH^f(\Bad{\multerror_-\psi}) = \infty$ but $\HH^f(\Bad{\multerror_+\psi}) = 0$. This stands in marked contrast to Theorem \ref{theoremJDV}, which demonstrates that $\HH^f(\Approx{\multerror\psi})$ is independent of $\multerror$ for all functions $f$ and $\psi$ and for all $\multerror > 0$. Going further, we show that the phenomenon of the above example is in fact commonplace:

\begin{corollary}
\label{corollarysurprising}
For any nice approximation function $\psi:(0,\infty)\to (0,\infty)$ such that $\psi/\psi_\ast \to 0$, there exists a dimension function $f$ and constants $0 < \multerror_- < \multerror_+$ such that
\[
\HH^f(\Bad{\multerror_+\psi}) = 0 \text{ but } \HH^f(\Bad{\multerror_-\psi}) = \infty.
\]
\end{corollary}
\begin{proof}
We define the dimension function $f$ via the formula
\[
f(\rho) =
\begin{cases}
\rho^\dimprod \exp F_\psi(\rho^{-\pdim/(\dimsum)}) & \rho \leq \rho_0\\
f(\rho_0) & \rho > \rho_0
\end{cases}
\]
where $\rho_0 > 0$ is a small constant. We know that $f$ is indeed a dimension function (i.e. is nondecreasing and satisfies $\lim_{\rho\to 0} f(\rho) = 0$) because $q^\qdim \psi^\pdim(q) \leq 1 < \qdim(\dimsum)$ for all $q$ sufficiently large, so if $\rho_0$ is small enough then $f$ is nondecreasing. This choice of $f$ guarantees that $L_{f,\multerror\psi} = \multerror^{-\pdim}$ for all $\multerror > 0$; choosing $0 < \multerror_- < \multerror_+$ so that $\multerror_+^{-\pdim} < \eta_\greeks < \multerror_-^{-\pdim}$ completes the proof.
\end{proof}

Corollary \ref{corollarysurprising} is less surprising than it first appears if we recall the contrast between the corollary of Dirichlet's theorem and the Jarn\'ik--Schmidt theorem: $\Approxmod{\psi_\ast}\norms = \MM$ (so that $\HD(\Badmod{\psi_\ast}\norms) = 0$), but $\HD(\BA) = \dimprod$ (so that $\HD(\Bad{\multerror\psi_\ast}) \to \dimprod$ as $\multerror \to 0$, and in particular $\HD(\Bad{\multerror\psi_\ast}) > 0$ for all sufficiently small $\multerror > 0$). Thus the conclusion of Corollary \ref{corollarysurprising} was already known to hold for one function, namely $\psi = \psi_\ast$.

Another surprising fact about \eqref{HfBpsi} is that it shows that $\HH^f(\Badmod\psi\greeks)$ depends on the norms $\|\cdot\|_\pgreek$ and $\|\cdot\|_\qgreek$ used to define $\Badmod\psi\greeks$. Actually, this is not a different fact than our previous surprising fact, since multiplying the $\|\cdot\|_\pgreek$ norm by a constant factor is equivalent to dividing $\psi$ by a constant factor; precisely,
\[
\Badmod{\multerror\psi}\greeks = \Badmod\psi{\pgreek',\qgreek} \text{ where }\|\cdot\|_{\pgreek'} = \|\cdot\|_\pgreek/\multerror.
\]
We conclude this introduction by unifying Theorems \ref{theoremHDasymp} and \ref{maintheorem} into a single theorem from which they can both be derived as corollaries.


\begin{theorem}
\label{maintheoremv2}
Let $f$ be a dimension function, and let $\psi:(0,\infty)\to (0,\infty)$ be a nice approximation function. Let $L_{f,\psi}$ be given by \eqref{Lfpsi}, and let $M_\psi = \psi^\pdim(1)$. Then there exist continuous functions $c_{\greeks},C_{\greeks}:\Rplus\to [0,\infty]$, depending only on $m$, $n$, $\firstgreek$, and $\secondgreek$, such that
\begin{equation}
\label{HfBpsiv2}
\HH^f(\Badmod\psi\greeks) = \begin{cases}
0 & \text{if } L_{f,\psi} < c_{\greeks}(M_\psi)\\
\infty & \text{if } L_{f,\psi} > C_{\greeks}(M_\psi)\\
\text{unknown} & \text{if } c_{\greeks}(M_\psi) \leq L_{f,\psi} \leq C_{\greeks}(M_\psi)
\end{cases}
\end{equation}
AND
\[
c_{\greeks}(0) = C_{\greeks}(0) = \eta_{\greeks}.
\]
\end{theorem}
\begin{proof}[Proof of Theorem \ref{theoremHDasymp} assuming Theorem \ref{maintheoremv2}]
Fix $s,\kappa > 0$, and let $f_s(\rho) = \rho^{\dimprod - s}$, $\psi_\kappa(q) = \kappa^{1/\pdim}\psi_\ast(q)$. Then
\begin{align*}
M_\psi &= \kappa&
F_\psi(Q) &= \sum_{q = 1}^Q \frac{\kappa}{q} \sim \kappa\log(Q)\\
\log\left(\frac{f(\rho)}{f_\ast(\rho)}\right) &= s|\log(\rho)|&
L_{f,\psi} &= \liminf_{\rho \to 0} \frac{s|\log(\rho)|}{\kappa\log(\rho^{-\pdim/(\dimsum)})} = \frac{s}{\kappa}\frac{\dimsum}{\pdim}
\end{align*}
and thus
\[
\HH^{\dimprod - s}(\Badmod\kappa\greeks)
= \HH^{f_s}(\Badmod{\psi_\kappa}\greeks)
= \begin{cases}
0 & \text{if } s < \kappa \frac{\pdim}{\dimsum} c_{\greeks}(\kappa)\\
\infty & \text{if } s > \kappa \frac{\pdim}{\dimsum} C_{\greeks}(\kappa)\\
\text{unknown} & \text{if } \kappa \frac{\pdim}{\dimsum} c_{\greeks}(\kappa) \leq s \leq \kappa \frac{\pdim}{\dimsum} C_{\greeks}(\kappa)
\end{cases}.
\]
It follows that
\[
\kappa \frac{\pdim}{\dimsum} c_{\greeks}(\kappa) \leq \dimprod - \HD(\Badmod\kappa\greeks) \leq \kappa \frac{\pdim}{\dimsum} C_{\greeks}(\kappa)
\]
and hence by the squeeze theorem,
\[
\lim_{\kappa\to 0} \frac{\dimprod - \HD(\Badmod\kappa\greeks)}{\kappa} = \frac{\pdim}{\dimsum} \eta_{\greeks} = \theta_{\greeks}.
\qedhere\]
\end{proof}
\begin{proof}[Proof of Theorem \ref{maintheorem} assuming Theorem \ref{maintheoremv2}]
Suppose that $\psi/\psi_\ast \to 0$. Then for each $\multerror > 0$, the function $\psi_\multerror = \psi\wedge\multerror\psi_\ast$ agrees with $\psi$ for all $q$ sufficiently large. It follows that $\Bad{\psi_\multerror} = \Bad\psi$. Taking the limit of $\eqref{HfBpsiv2}_{\psi = \psi_\multerror}$ as $\multerror\to 0$ yields \eqref{HfBpsi}.
\end{proof}


The remainder of the paper is devoted to giving first a heuristic argument for Theorem \ref{maintheoremv2}, and then a rigorous proof motivated by the heuristic argument.

{\bf Note.} To avoid needlessly cluttering the notation, in what follows we usually omit the subscripts of norms, writing $\|\pp\|$ for $\|\pp\|_\pgreek$, $\|\qq\|$ for $\|\qq\|_\qgreek$, etc. When it is important to clarify or emphasize which norm is being used, we will do so.\\

{\bf Acknowledgements.} The author would like to thank Faustin Adiceam for helpful comments on an earlier version of this paper. The author was supported in part by the EPSRC Programme Grant EP/J018260/1. The author thanks the anonymous referee for many helpful comments.

\tableofcontents

\section{Hausdorff measure and dimension}
\label{sectionHD}
A function $f:(0,\infty)\to (0,\infty)$ is called a \emph{dimension function} if $f$ is nondecreasing and $\lim_{\rho \to 0} f(\rho) = 0$.  In this section, we fix a dimension function $f$ and a set $S\subset \R^\generic$. The \emph{Hausdorff $f$-measure} of a set $S\subset\R^\generic$ is defined by the formula
\[
\HH^f(S) = \lim_{\epsilon\to 0}\inf\left\{\sum_{i = 1}^\infty f(\rho_i): \text{$\big(\ball{x_i}{\rho_i}\big)_1^\infty$ is a countable cover of $S$ with $\rho_i\leq\epsilon \all i$}\right\},
\]
where here and hereafter, $\ball x\rho = \ballvar x\rho \df \{y : \dist(x,y) \leq \rho\}$. Finally, the \emph{Hausdorff dimension} of $S$ is given by the formula
\[
\HD(S) = \inf\{s\geq 0 : \HH^s(S) = 0\},
\]
where $\HH^s$ is the Hausdorff measure with respect to the dimension function
\[
f_s(\rho) = \rho^s.
\]
We recall some well-known upper and lower bounds on the quantities $\HH^f(S)$ and $\HD(S)$. The upper bounds come from considering only covers whose balls all have the same radius. Such a cover is represented by two quantities: the set $F$ of centers of balls in the cover, and the common radius $\rho$. The cost of such a cover is equal to $\#(F) f(\rho)$. It follows that
\[
\HH^f(S) \leq \underline\BB^f(S) \df \liminf_{\rho\to 0} \boxc_\rho(S) f(\rho),
\]
where
\begin{equation}
\label{boxcdef}
\boxc_\rho(S) = \min\left\{\#(F) : F\subset \R^\generic, \; S\subset \bigcup_{x\in F} \ball x\rho\right\}.
\end{equation}
We call the number $\underline\BB^f(S)$ the \emph{lower box-counting $f$-measure}\Footnote{Warning: Although we are calling $\underline\BB^f$ a ``measure'', it is not countably additive on any reasonable collection of sets, although it is finitely additive on sets $A,B$ such that $\dist(A,B) > 0$.} of $S$, in analogy with the \emph{lower box-counting dimension}, which is given by the formula
\[
\underline\BD(S) = \liminf_{\rho \to 0} \frac{\log \boxc_\rho(S)}{-\log(\rho)} = \inf\{s\geq 0 : \underline\BB^s(S) = 0\}.
\]
Note that
\[
\HD(S) \leq \underline\BD(S).
\]
To get a lower bound on $\HH^f(S)$ and $\HD(S)$, a more sophisticated argument is needed.
\begin{lemma}[Generalized mass distribution principle]
\label{lemmamassdistribution}
Let $\mu$ be a probability measure on $S$ such that for all $\xx\in\R^\generic$ and $\rho > 0$,
\[
\mu(\ball \xx\rho) \leq g(\rho),
\]
where $g$ is a dimension function. Then
\[
\HH^f(S) \geq \liminf_{\rho\to 0} \frac{f(\rho)}{g(\rho)}\cdot
\]
\end{lemma}
\begin{proof}
This can be deduced as a corollary of \cite[Theorem 8.2]{MSU}, but the proof is much easier. If $(\ball{\xx_i}{\rho_i})_1^\infty$ is a cover of $S$ such that $\rho_i \leq \epsilon \all i$, then
\[
1 = \mu(S) \leq \sum_{i = 1}^\infty \mu(\ball{\xx_i}{\rho_i}) \leq \sum_{i = 1}^\infty g(\rho_i) \leq \left(\sup_{0 < \rho \leq \epsilon} \frac{g(\rho)}{f(\rho)}\right) \sum_{i = 1}^\infty f(\rho_i)
\]
and thus
\[
\HH^f(S) \geq \lim_{\epsilon\to 0} \inf_{0 < \rho\leq \epsilon} \frac{f(\rho)}{g(\rho)} = \liminf_{\rho\to 0} \frac{f(\rho)}{g(\rho)}\cdot
\qedhere\]
\end{proof}

\draftnewpage
\section{Heuristic argument}
\label{sectionheuristic}
In this section we give a heuristic argument explaining where the formulas \eqref{thetamn} and \eqref{HfBpsi} come from. Let $f$ be a dimension function, and let $\psi$ be a nice approximation function (cf. Definition \ref{definitionniceapproximation}). The reader can keep in mind the special case $f(\rho) = \rho^s$, $\psi(q) = \kappa^{1/\pdim} \psi_\ast(q) = (\kappa q^{-\qdim})^{1/\pdim}$, where $0 < s < \dimprod$ and $\kappa > 0$, i.e. the special case needed to prove Theorem \ref{theoremHDasymp}.

In this section, the notation $A\sim B$ indicates that the ratio $A/B$ is asymptotically close to 1, but we do not specify the circumstances under which this asymptotic holds. Similarly, $A\approx B$ indicates that the ratio $A/B$ is bounded from above and below, but again the contextual requirements are unspecified. A more rigorous way of denoting asymptotic relations will be described at the beginning of Section \ref{sectionpreliminaries}.

The first step in our calculation is to rewrite the set $\Bad\psi$ as the union of countably many closed sets. For each $Q_0\in\N$, let
\begin{align}
\label{BpsiQdef}
W_{\psi,Q_0} &\df \bigcup_{\substack{(\pp,\qq)\in\Z^\dimsum \\ \|\qq\|\geq Q_0}} \Delta_\psi(\pp,\qq),&
B_{\psi,Q_0} &\df \MM\butnot W_{\psi,Q_0},
\end{align}
where for each $\pp\in\R^\pdim$ and $\qq\in\R^\qdim\butnot\{\0\}$ we have
\begin{equation}
\label{Deltapsidef}
\Delta_\psi(\pp,\qq) = \left\{A\in \MM: \|A\qq - \pp\| \leq \psi(\|\qq\|)\right\}.
\end{equation}
Then we have
\[
\Bad\psi = \bigcup_{Q_0 = 1}^\infty B_{\psi,Q_0},
\]
so
\begin{equation}
\label{supoverQ0}
\HH^f(\Bad\psi) = \sup_{Q_0\geq 1} \HH^f(B_{\psi,Q_0}) = \infty\cdot\sup_{Q_0\geq 1} \HH^f(B_{\psi,Q_0}\cap\KK),\Footnote{Here we use the convention that $\infty\cdot 0 = 0$.}
\end{equation}
where $\KK$ denotes the set of $\pdim\times\qdim$ matrices whose entries are all in $[0,1]$, i.e. the ``unit cube'' of $\MM$.

Fix $Q_0\in\N$, and we will estimate $\HH^f(B_{\psi,Q_0}\cap\KK)$. We will use the lower box-counting measure $\underline\BB^f(B_{\psi,Q_0}\cap\KK)$ (cf. Section \ref{sectionHD}) as a proxy for $\HH^f(B_{\psi,Q_0}\cap\KK)$. Intuitively, $\underline\BB^f(B_{\psi,Q_0}\cap\KK)$ is a good proxy for $\HH^f(B_{\psi,Q_0}\cap\KK)$ as long as the closed set $B_{\psi,Q_0}\cap\KK$ is sufficiently homogeneous.

\begin{heur}
\label{heur1}
For all sufficiently large $Q_0\in\N$, we have $\HH^f(B_{\psi,Q_0}\cap\KK) \approx \underline\BB^f(B_{\psi,Q_0}\cap\KK)$.
\end{heur}

Now to compute $\underline\BB^f(B_{\psi,Q_0}\cap \KK)$, we should fix $\rho > 0$ and estimate the number $\boxc_\rho(B_{\psi,Q_0}\cap \KK)$ defined in \eqref{boxcdef}. To do this, we will estimate the measure of the set $\thickvar{B_{\psi,Q_0}}\rho\cap\KK$ in order to apply the formula
\[
\boxc_\rho(B_{\psi,Q_0}\cap \KK) \approx \rho^{-\dimprod} \lebesgue_\KK\big(\thickvar{B_{\psi,Q_0}}\rho\big).
\]
Here $\lebesgue_\KK$ denotes Lebesgue measure on $\KK$, normalized to be a probability measure, and
\[
\thickvar S\rho \df \{A\in\MM : \dist(A,S) \leq \rho\}.
\]
(The distance on $\MM$ comes from the operator norm $\|A\| = \max_{\|\qq\|_\qgreek = 1} \|A\qq\|_\pgreek$.) Fix $Q\geq Q_0$ to be determined, and let
\begin{align}
\label{BpsiQ0Qdef}
W_\psi(Q_0,Q) &= \bigcup_{\substack{(\pp,\qq)\in\Z^\dimsum \\ Q_0 \leq \|\qq\| \leq Q}} \Delta_\psi(\pp,\qq),&
B_\psi(Q_0,Q) &= \MM\butnot W_\psi(Q_0,Q).
\end{align}
We will estimate the measure of $B_\psi(Q_0,Q)$ using an independence assumption, and then we will compare the measure of $B_\psi(Q_0,Q)$ with the measure of $\thickvar{B_{\psi,Q_0}}\rho$, assuming an appropriate relation between $Q$ and $\rho$. However, we must deal with the fact that if $(\pp,\qq)\in\Z^\dimsum$ and $k\in\Z\butnot\{0\}$, then $\Delta_\psi(k\pp,k\qq) \subset \Delta_\psi(\pp,\qq)$, and in particular the sets $\Delta_\psi(\pp,\qq)$ and $\Delta_\psi(k\pp,k\qq)$ are not independent. Let $\PP\subset\Z^\pdim\times(\Z^\qdim\butnot\{\0\})$ be a set consisting of primitive integer vectors such that for all $(\pp,\qq)\in\Z^\pdim\times(\Z^\qdim\butnot\{\0\})$, we have $(\pp,\qq)\in k \PP$ for exactly one value $k\in\Z\butnot\{0\}$.
\begin{heur}
\label{heur2}
The sets $\Delta_\psi(\pp,\qq)$ ($(\pp,\qq)\in \PP$, $Q_0\leq \|\qq\|\leq Q$) are approximately independent with respect to the probability measure $\lebesgue_\KK$.
\end{heur}
Using this assumption, we make the following calculation:
\begin{equation*}
\begin{split}
-\log\lebesgue_\KK\big(B_\psi(Q_0,Q)\big) &\sim \sum_{\substack{(\pp,\qq)\in \PP \\ Q_0 \leq \|\qq\| \leq Q}} -\log\big(1 - \lebesgue_\KK(\Delta_\psi(\pp,\qq))\big)\\
&\sim \sum_{\substack{(\pp,\qq)\in \PP \\ Q_0 \leq \|\qq\| \leq Q}} \lebesgue_\KK\big(\Delta_\psi(\pp,\qq)\big).
\end{split}
\end{equation*}
\begin{heur}
The set $\PP$ is ``homogeneous'' in the sense that sums over subsets of $\PP$ are asymptotic to a fixed constant, say $\multerror_\PP$, times integrals over the same regions.
\end{heur}
To compute $\multerror_\PP$, we observe that for $R$ large,
\begin{align*}
\Vgreeks R^\dimsum &\sim \#\{(\pp,\qq) \in \Z^\pdim \times (\Z^\qdim \butnot \{\0\}): \|\pp\|\vee\|\qq\| \leq R\}\\
&= \sum_{k\in\Z\butnot\{0\}} \#\{(\pp,\qq)\in \PP : \|k\pp\|\vee\|k\qq\| \leq R\}\\
&\sim \multerror_\PP \sum_{k\in\Z\butnot\{0\}} \Vgreeks (R/|k|)^\dimsum = 2\zeta(\dimsum) \multerror_\PP \Vgreeks R^\dimsum,
\end{align*}
so $\multerror_\PP = 1/(2\zeta(\dimsum))$. Thus,
\begin{align*}
&-\log\lebesgue_\KK\big(B_\psi(Q_0,Q)\big)\\
&\sim \frac{1}{2\zeta(\dimsum)} \int_{\substack{(\pp,\qq)\in\R^{\dimsum} \\ Q_0 \leq \|\qq\| \leq Q}} \lebesgue_\KK\big(\Delta_\psi(\pp,\qq)\big) \;\dee(\pp,\qq)\\
&= \frac{1}{2\zeta(\dimsum)} \int_{Q_0 \leq \|\qq\| \leq Q} \int_\KK \lebesgue_{\R^\pdim}\big(\big\{\pp\in\R^\pdim : \|A\qq - \pp\| \leq \psi(\|\qq\|)\big\}\big) \;\dee A \;\dee\qq\\
&= \frac{1}{2\zeta(\dimsum)} \int_{Q_0 \leq \|\qq\| \leq Q} V_\pgreek \psi^\pdim(\|\qq\|) \;\dee\qq.
\end{align*}
%
\begin{lemma}
\label{lemmasphericalcoordinates}
If $f:\Rplus\to\Rplus$ is locally integrable, then
\[
\int f(\|\qq\|) \;\dee\qq = \qdim V_\qgreek \int_0^\infty q^{\qdim - 1} f(q) \;\dee q.
\]
\end{lemma}
\begin{proof}
It suffices to prove the equality for functions of the form $f(q) = \big[q \leq Q\big]$ (cf. Convention \ref{conventioniverson} below), and for these,
\[
\int \big[\|\qq\| \leq Q\big] \;\dee\qq = \lebesgue_{\R^\qdim}\big(B_\qgreek(\0,Q)\big) = V_\qgreek Q^\qdim = V_\qgreek \int_0^Q \qdim q^{\qdim - 1} \;\dee q.
\qedhere\]
\end{proof}
\noindent Using Lemma \ref{lemmasphericalcoordinates}, we continue the calculation:
\begin{equation}
\label{FpsiQheur}
\begin{split}
-\log\lebesgue_\KK\big(B_\psi(Q_0,Q)\big)
&\sim \frac{V_\pgreek}{2\zeta(\dimsum)} \int_{Q_0 \leq \|\qq\| \leq Q} \psi^\pdim(\|\qq\|) \;\dee\qq\\
&= \frac{\qdim \Vgreeks}{2\zeta(\dimsum)} \int_{Q_0}^Q q^{\qdim - 1} \psi^\pdim(q) \;\dee q\\
&\sim \eta_\greeks F_\psi(Q),
\end{split}
\end{equation}
where $F_\psi$ is as in \eqref{Fpsidef1}.

Now that we have estimated the measure of $B_\psi(Q_0,Q)$, we want to compare it to the measure of $\thickvar{B_{\psi,Q_0}}\rho$. To do this, we compute the ``thickness'' of the sets $\Delta_\psi(\pp,\qq)$ ($(\pp,\qq)\in \PP$). The following lemma (which will be used in the real proof of Theorem \ref{maintheoremv2}) suggests that the thickness should be interpreted as being $\Psi(\|\qq\|) \df \psi(\|\qq\|)/\|\qq\|$:
\begin{lemma}
\label{lemmathickness}
Let $\psi_1,\psi_2:\N\to (0,\infty)$. Then for all $\pp \in \R^\pdim$ and $\qq \in \R^\qdim \butnot \{\0\}$,
\[
\NN\left(\Delta_{\psi_1}(\pp,\qq),\frac{\psi_2(\|\qq\|)}{\|\qq\|}\right) \subset \Delta_{\psi_1 + \psi_2}(\pp,\qq).
\]
\end{lemma}
For example, $\Delta_0(\pp,\qq)$ is an affine linear subspace of $\MM$, and $\thickvar{\Delta_0(\pp,\qq)}{\Psi(\|\qq\|)} \subset \Delta_\psi(\pp,\qq)$.
\begin{proof}
Fix $A\in \Delta_{\psi_1}(\pp,\qq)$ and $B\in \MM$ with $\|B - A\| \leq \psi_2(\|\qq\|)/\|\qq\|$. Then
\begin{align*}
\|B\qq - \pp\|
&\leq \|A\qq - \pp\| + \|B - A\|\cdot \|\qq\|\\
&\leq \psi_1(\|\qq\|) + \frac{\psi_2(\|\qq\|)}{\|\qq\|}\cdot\|\qq\| = (\psi_1 + \psi_2)(\|\qq\|).
\qedhere\end{align*}
\end{proof}

\begin{heur}
\label{heur3}
The measure of $B_\psi(Q_0,Q)$ is approximately the same as the measure of $\thickvar{B_{\psi,Q_0}}\rho$, where $\rho = \Psi(Q)$ is the minimum ``thickness'' of the sets $\Delta_\psi(\pp,\qq)$ ($Q_0 \leq \|\qq\| \leq Q$).
\end{heur}


So
\[
-\log\lebesgue_\KK\big(\thickvar{B_{\psi,Q_0}}\rho\big) \sim \eta_\greeks F_\psi(\Psi^{-1}(\rho)).
\]
Since $\psi$ is assumed to be a nice approximation function, we have $\log\psi(q) \sim -\frac{\qdim}{\pdim}\log(q)$ and thus $\log\Psi(q) \sim -\frac{\dimsum}{\pdim}\log(q)$ and $\log\Psi^{-1}(\rho) \sim -\frac{\pdim}{\dimsum}\log(\rho)$. Since $F_\psi$ is either logarithmic or sublogarithmic, this gives
\[
-\log\lebesgue_\KK\big(\thickvar{B_{\psi,Q_0}}\rho\big) \sim \eta_\greeks F_\psi(\rho^{-\pdim/(\dimsum)}).
\]
We now have all the ingredients needed to finish the calculation. In what follows, the symbol $\sim$ just means intuitively that ``the quantities are close'' and we do not specify a more precise relation.
\begin{align*}
\log\underline\BB^f(B_{\psi,Q_0}\cap \KK)
&= \liminf_{\rho\to 0}\log\big(\boxc_\rho(B_{\psi,Q_0})f(\rho)\big)\\
&\sim \liminf_{\rho\to 0}\big[\log\lebesgue_\KK\big(\thickvar{B_{\psi,Q_0}}\rho\big) - \dimprod\log(\rho) + \log f(\rho)\big]\\
&\sim \liminf_{\rho\to 0}\left[-\eta_\greeks F_\psi(\rho^{-\pdim/(\dimsum)}) + \log\frac{f(\rho)}{f_\ast(\rho)}\right]\\
&= \begin{cases}
-\infty & \text{if } L_{f,\psi} < \eta_\greeks\\
\infty & \text{if } L_{f,\psi} > \eta_\greeks\\
\text{unknown} & \text{if } L_{f,\psi} = \eta_\greeks
\end{cases}.
\end{align*}
Using heuristic assumption \ref{heur1} and applying \eqref{supoverQ0} gives
\[
\HH^f(\Bad\psi)
= \begin{cases}
0 & \text{if } L_{f,\psi} < \eta_\greeks\\
\infty & \text{if } L_{f,\psi} > \eta_\greeks\\
\text{unknown} & \text{if } L_{f,\psi} = \eta_\greeks
\end{cases}.
\]

\draftnewpage
\section{Preliminaries and notation}
\label{sectionpreliminaries}

The notation introduced in this section will be used throughout the paper. We provide a summary in Appendix \ref{sectionnotation} for convenience. Recall that $\dims\in\N$ are fixed, and that $\MM$ denotes the set of $\pdim\times \qdim$ matrices.

\begin{convention}
The symbols $\lessapprox$, $\gtrapprox$, and $\approx$ will denote coarse multiplicative asymptotics, with a subscript of $\plus$ indicating that the asymptotic is additive instead of multiplicative, and other subscripts indicating variables that the implied constant is allowed to depend on. If there are no variables indicated, then the implied constant only depends on the parameters $\firstdim$, $\seconddim$, $\firstgreek$, and $\secondgreek$. The implied constant is understood to depend continuously on all parameters. We emphasize that all parameters the implied constant depends on other than $\firstdim$, $\seconddim$, $\firstgreek$, and $\secondgreek$ will be explicitly notated, even if they appear to be clear from context.

For example, $A\lessapprox_K B$ means that there exists a constant $k \geq 1$ that depends continuously on the parameters $\firstdim$, $\seconddim$, $\firstgreek$, $\secondgreek$, and $K$, and on no other parameters, such that $A \leq k B$. Similarly, $A\approx_\plus B$ means that there exists a constant $k\geq 0$, depending only on $\firstdim$, $\seconddim$, $\firstgreek$, and $\secondgreek$, such that $B - k \leq A \leq B + k$. Note that in the first example, if we have an upper bound and a lower bound on $K$ then we can use continuity to get an upper bound on the implied constant $k$.
%
\end{convention}

\begin{convention}
\label{conventionconvergence}
When limits are written using arrow notation, they are assumed to be uniform with respect to all parameters except for those mentioned in parentheses. For example, the notation
\[
A \tendsto{\substack{Q\to \infty \\ \delta \to 0 \; (K)}} 0
\]
means: For all $\epsilon > 0$, ther exist $Q_0,\delta_0 > 0$, with $Q_0$ depending only on $\epsilon$ and $\delta_0$ depending on both $K$ and $\epsilon$ (and on $\constants$), such that if $Q \geq Q_0$ and $0 < \delta \leq \delta_0$, then $|A| \leq \epsilon$.
\end{convention}

\begin{convention}
The symbols $\sim$, $\lesssim$, and $\gtrsim$ will denote asymptotics whose implied constants tend to 1 (or 0 for additive asymptotics) as certain quantities approach their limits. The convergence of the implied constant is notated using Convention \ref{conventionconvergence}. For example, the formula
\[
A \eqsim{\delta \to 0 \; (Q)} B
\]
means: For all $\epsilon > 0$, there exists $\delta_0 > 0$ depending only on $\epsilon$, $Q$, and $\constants$ such that if $0 < \delta \leq \delta_0$, then $(1 - \epsilon) B \leq A \leq (1 + \epsilon) B$.

Similarly, the notation $A \lessless B$ means that $A/B \to 0$; again the convergence is notated underneath the $\lessless$ sign using Convention \ref{conventionconvergence}. (This differs from the common usage of $\lessless$ as the Vinogradov symbol; in this paper, $\lessapprox$ plays the function of the Vinogradov symbol.)
\end{convention}

\begin{convention}
Whenever a vector $\rr\in \R^\dimsum$ is fixed, we denote its components by $\pp\in\R^\pdim$ and $\qq\in\R^\qdim$, so that $\rr = (\pp,\qq)$. Similarly, whenever $\pp\in\R^\pdim$ and $\qq\in\R^\qdim$ are fixed, we use the shorthand $\rr = (\pp,\qq)$. The same convention applies to $\rr' = (\pp',\qq')$.
\end{convention}

\begin{convention}
\label{conventioniverson}
We use the Iverson bracket notation $[\text{statement}] = \begin{cases} 1 & \text{statement true} \\ 0 & \text{statement false} \end{cases}$.
\end{convention}

\begin{convention}
The symbol $\triangleleft$ will be used to indicate the end of a nested proof.
\end{convention}


\subsection{Cantor series expansion}
\label{subsectioncantorseries}
Let $\Dimprod = \dimprod$, so that $\R^\Dimprod$ is isomorphic to $\MM$ as a vector space. In what follows, we identify an element of $\MM$ with its image in $\R^\Dimprod$ under the map which sends a matrix to the list of its entries, and we take the norm on $\MM$ to be the operator norm $\|A\| = \max_{\|\qq\| = 1} \|A\qq\|$. We write $\KK = [0,1]^\Dimprod \subset \MM$. Note that this subsection and the next one (i.e. \6\ref{subsectioncantorseries}-\ref{subsectionevaporationrates}) make sense in any finite-dimensional normed space $(\R^\Dimprod,\|\cdot\|)$, without any reference to Diophantine approximation.

During the proof of Theorem \ref{maintheoremv2}, we will specify a sequence of integers $(N_k)_1^\infty$, depending on the approximation function $\psi$ and on an auxiliary parameter $\beta$. We will describe how this sequence is chosen later, but for now we assume that we are given a sequence $(N_k)_1^\infty$ satisfying $N_k\geq 2 \all k$. We introduce the following notations:
\begin{align*}
E_k &\df \{0,\ldots,N_k - 1\}^\Dimprod \subset \MM \note{$1 \leq k < \infty$}\\
E^k &\df \prod_{j = 1}^k E_j \;\;\; (\text{Cartesian product}) \note{$0 \leq k \leq \infty$}\\
N^k &\df \prod_{j = 1}^k N_j = \#(E^k)^{1/\Dimprod }. \note{$0 \leq k < \infty$}
\end{align*}

{\bf Warning.} The notations $E^k$ and $N^k$ should not be confused with the Cartesian and arithmetic powers of a set $E$ and a number $N$, respectively. In this paper, there is no set $E$ and there is no number $N$. As a plea to forgive this abuse of notation, we note that in the special case occurring in the proof of Theorem \ref{theoremHDasymp}, the sequence $(N_k)_1^\infty$ is constant (cf. Remark \ref{remarkNkconstant} below), and so in that case the notations $E^k$ and $N^k$ really do denote Cartesian and arithmetic powers.

For each $\omega\in E^k$, we write $|\omega| = k$. Let $E^* \df \bigcup_{k\geq 0} E^k$. ($E^\infty$ is not included in this union.) For all $\omega\in E^*\cup E^\infty$ and $k\leq |\omega|$, we denote the $k$th element of $\omega$ by $\omega_k$, and we denote the initial segment of $\omega$ of length $k$ by $\omega\given k$. We define the \emph{coding map} $\pi:E^*\cup E^\infty\to \KK$ via the formula
\[
\pi(\omega) = \sum_{k = 1}^{|\omega|} \frac{\omega_k}{N^k}\cdot
\]
This series is called the \emph{Cantor series expansion} of $\pi(\omega)$. Note that every element of $\KK$ has at least one Cantor series expansion, and almost every element of $\KK$ (with respect to Lebesgue measure) has exactly one Cantor series expansion. Moreover, for each $\omega\in E^*$, if
\begin{align*}
[\omega] &\df \{\tau \in E^\infty : \tau\given |\omega| = \omega\},&
\KK_\omega &\df \pi([\omega]),
\end{align*}
then
\begin{equation}
\label{pi[omega]}
\KK_\omega = \pi(\omega) + \frac{1}{N^{|\omega|}}\KK.
\end{equation}
In particular, $\KK_\omega$ is a cube of side length $1/N^{|\omega|}$. We observe that if $\omega,\tau\in E^*$ are incomparable,\Footnote{Two strings are called \emph{incomparable} if neither is an initial segment of the other.} then $\KK_\omega$ and $\KK_\tau$ have disjoint interiors.

The following notation will sometimes be convenient: Given $\omega\in E^*$ and $A\in\KK$, let
\[
\Phi_\omega(A) = \pi(\omega) + \frac{A}{N^{|\omega|}}\cdot
\]
By \eqref{pi[omega]}, $\Phi_\omega:\KK\to\KK_\omega$ is a bijection.

\subsection{Evaporation rates}
\label{subsectionevaporationrates}
A \emph{tree} in $E^*$ is a subset of $E^*$ which is closed under taking initial segments. Let $T^*\subset E^*$ be a tree. For each $k\in\N$, let $T^k \df E^k \cap T^*$, and for each $\omega\in T^*$, let
\[
T_\omega = \{a\in E_{|\omega| + 1} : \omega a \in T^{|\omega| + 1}\}.
\]
Here and hereafter, $\omega a$ denotes the concatenation of $\omega$ and $a$. For each $k\geq 1$, let
\begin{align*}
P_k^+ &\df \frac{1}{(N_k)^\Dimprod}\max_{\omega\in T^{k - 1}} \#(E_k\butnot T_\omega)\\
P_k^- &\df \frac{1}{(N_k)^\Dimprod}\min_{\omega\in T^{k - 1}} \#(E_k\butnot T_\omega).
\end{align*}
In other words, $P_k^+$ (resp. $P_k^-$) is the maximum (resp. minimum) proportion of children that get removed from the tree $T$ at stage $k$. We call the sequences $(P_k^+)_1^\infty$ and  $(P_k^-)_1^\infty$ the \emph{upper evaporation rate} and the \emph{lower evaporation rate} of $T^*$, respectively.

Let $T^\infty$ denote the set of infinite branches through $T^*$, i.e. the set $\{\omega\in E^\infty : \omega\given k \in T^k\all k\}$. The goal of this subsection is to relate the Hausdorff $f$-measure of $\pi(T^\infty)$ to the evaporation rate of $T^*$, where $f$ is a dimension function. To do this, we introduce some more notation. Namely, let
\begin{align*}
M_k^- &\df (N_k)^\Dimprod (1 - P_k^+) = \min_{\omega\in T^{k - 1}} \#(T_\omega)\\
M_k^+ &\df (N_k)^\Dimprod (1 - P_k^-) = \max_{\omega\in T^{k - 1}} \#(T_\omega)\\
M_\pm^k &\df \prod_{j = 1}^k M_j^\pm.
\end{align*}
In other words, $M_k^-$ (resp. $M_k^+$) is the minimum (resp. maximum) number of children of any stage $k - 1$ node that remain in $T$. For each $0 < \rho \leq 1$, let $k(\rho)\in\N$ be the unique integer satisfying $1/N^{k(\rho) + 1} < \rho \leq 1/N^{k(\rho)}$, and let
\[
f_\pm(\rho) \df \frac{(N^{k(\rho)}\rho)^\Dimprod}{M_\mp^{k(\rho)}} = \left(\prod_{j = 1}^{k(\rho)} \frac{1}{1 - P_j^\pm}\right)\rho^\Dimprod.
\]
\begin{remark*}
The subscripts and superscripts above have been chosen so that $\sq_- \leq \sq_+$ for all relevant objects.
\end{remark*}

\begin{proposition}
\label{propositionevaporating}
Let $f$ be a dimension function. With notation as above,
\begin{itemize}
\item[(i)] If $T^* \subset E^*$ is a tree such that $\emptyset\in T^*$ and $\sup_{k\geq 1} P_k^+ < 1$, then
\begin{equation}
\label{evaporating1}
\HH^f(\pi(T^\infty)) \gtrapprox \liminf_{\rho\to 0}\frac{f(\rho)}{f_+(\rho)}\cdot
\end{equation}
\item[(ii)] For any tree $T^*\subset E^*$,
\begin{equation}
\label{evaporating2}
\HH^f(\pi(T^\infty)) \leq \underline\BB^f(\pi(T^\infty)) \lessapprox \liminf_{\rho\to 0}\frac{f(\rho)}{f_-(\rho)}\cdot
\end{equation}
\end{itemize}
\end{proposition}
\begin{proof}[Proof of \text{(i)}]
By pruning the tree $T^*$, we may without loss of generality assume that $\#(T_\omega) = M_{|\omega| + 1}^-$ for all $\omega\in T^*$. Let $\nu$ be the unique measure on $T^\infty$ such that
\[
\nu([\omega]) = 1/M_-^{|\omega|} \all \omega\in T^*.
\]
Let $\mu$ be the image of $\nu$ under the coding map $\pi$. Now fix $A\in \KK$ and $0 < \rho \leq 1$, and let $k = k(\rho)$. Let $\diamcube = \max_{B\in\KK} \|B\|$. Then
\begin{align*}
\mu(\ball A\rho)
&= \sum_{\omega\in E^{k + 1}}\nu\big([\omega]\cap \pi^{-1}(\ball A\rho)\big)\\
&\leq \frac{1}{M_-^{k + 1}}\#\{\omega\in E^{k + 1} : \KK_\omega\cap \ball A\rho\neq\emptyset\} \noreason\\
&\leq \frac{1}{M_-^{k + 1}}\#\{\omega\in E^{k + 1} : \KK_\omega\subset \ballvar A{(1 + \diamcube)\rho}\} \\ \since{$\diam(\KK_\omega) \leq \diamcube/N^{k + 1} < \diamcube \rho$ for all $\omega\in E^{k + 1}$}\\
&\leq \frac{1}{M_-^{k + 1}}\frac{\lebesgue_\KK\big(\ballvar A{(1 + \diamcube)\rho}\big)}{(1/N^{k + 1})^\Dimprod} \\ \since{$\lebesgue_\KK(\KK_\omega) = (1/N^{k + 1})^\Dimprod$ for all $\omega\in E^{k + 1}$}\\
&\approx \frac{(N^{k + 1})^\Dimprod}{M_-^{k + 1}}\rho^\Dimprod\\
&\approx \frac{(N^k)^\Dimprod}{M_-^k}\rho^\Dimprod \\ \since{$\sup_{k\geq 1} P_k^+ < 1$}\\
&= f_+(\rho).
\end{align*}
Applying Lemma \ref{lemmamassdistribution} demonstrates \eqref{evaporating1}.
\end{proof}

\begin{proof}[Proof of \text{(ii)}]
Fix $0 < \rho \leq 1$, and let $k = k(\rho)$. For each $\omega\in T^k$, we have $\boxc_\rho(\KK_\omega) = \boxc_{N^k\rho}(\KK) \approx 1/(N^k\rho)^\Dimprod$. Since $\#(T^k) \leq M_+^k$, it follows that
\[
\boxc_\rho(\pi(T^\infty)) \leq \sum_{\omega\in T^k} \boxc_\rho(\KK_\omega) \lessapprox \frac{M_+^k}{(N^k\rho)^\Dimprod} = \frac{1}{f_-(\rho)}\cdot
\]
Multiplying by $f(\rho)$ and taking the liminf as $\rho\to 0$ finishes the proof.
\end{proof}

\subsection{Choice of the sequence $(N_k)_1^\infty$}
In this subsection we describe how the sequence $(N_k)_1^\infty$ to be used in the proof of Theorem \ref{maintheoremv2} will be chosen, depending on a nice approximation function $\psi$ and a parameter $\beta > 0$. To motivate this choice, we recall that in the heuristic argument of Section \ref{sectionheuristic}, we estimated (cf. \eqref{FpsiQheur}) that for all $1 \leq Q_1 \leq Q_2$,
\begin{equation}
\label{heuristicFpsiQ1Q2}
-\log\lebesgue_\KK\big(B_\psi(Q_1,Q_2)\big) \sim \eta_\greeks F_\psi(Q_1,Q_2),
\end{equation}
where $B_\psi(Q_1,Q_2)$ is defined by \eqref{BpsiQ0Qdef}, and
\begin{equation}
\label{Fpsidef2}
F_\psi(Q_1,Q_2) \df \int_{Q_1}^{Q_2} q^{\qdim - 1} \psi^\pdim(q) \;\dee q.
\end{equation}
Note that the two-input function $F_\psi$ defined here is related to the one-input function $F_\psi$ defined in \eqref{Fpsidef1} via the asymptotic
\[
F_\psi(Q) \eqsim{Q\to \infty \; (Q_0)} F_\psi(Q_0,Q).
\]
We have switched from a sum to an integral for ease of calculations later.

Coming back to the motivation of the choice of $(N_k)_1^\infty$, we want to choose it so that the companion sequence
\begin{equation}
\label{Qkdef}
Q^k \df (N^k)^{\pdim/(\dimsum)} \;\;\; (k\in\N)
\end{equation}
satisfies $F_\psi(Q^k,Q^{k + 1}) \sim \beta$ for all $k$, for an appropriate limiting process. The idea is to make sure that $Q^k$ and $Q^{k + 1}$ are ``far enough apart'' that the heuristic formula \eqref{heuristicFpsiQ1Q2} is true for $Q_1 = Q^k$ and $Q_2 = Q^{k + 1}$, but ``close enough together'' that the left-hand side of \eqref{heuristicFpsiQ1Q2} is essentially the same as $\lebesgue_\KK\big(W_\psi(Q^k,Q^{k + 1})\big)$. The significance of the sequence \eqref{Qkdef} is that the ``minimum thickness'' of $B_\psi(Q^k,Q^{k + 1})$ is approximately $1/N^{k + 1}$ (cf. Lemma \ref{lemmathickness} and Heuristic assumption \ref{heur3}) and thus heuristically, we expect that if $\omega\in E^\infty$ is the Cantor series expansion of a matrix $A = \pi(\omega) \in \KK$, then it should be possible to control whether or not $A \in B_\psi(Q^k,Q^{k + 1})$ by controlling the coordinate $\omega_{k + 1}$, given information about the initial segment $\omega\given k$. This will allow us to create a tree contained in $B_{\psi,Q_0}$ that we can control the evaporation rate of.

Now we need to show that it is possible to choose the sequence $(N_k)_1^\infty$ in such a way so as to satisfy $F_\psi(Q^k,Q^{k + 1}) \sim \beta \all k$, or more precisely,
\begin{equation}
\label{Fpsisim}
F_\psi(Q^k,Q^{k + 1}) \eqsim{M_\psi/\beta \to 0} \beta.
\end{equation}
Before we do, we collect some facts about the function $F_\psi$. Since $\psi$ is a nice approximation function, the function
\begin{equation}
\label{phidef}
\phi(q) \df \left(\frac{\psi(q)}{\psi_\ast(q)}\right)^\pdim = q^\qdim \psi^\pdim(q)
\end{equation}
is nonincreasing, and thus $\phi(Q_2) \leq \phi(q) \leq \phi(Q_1)$ for all $q\in [Q_1,Q_2]$. Since
\[
F_\psi(Q_1,Q_2) = \int_{Q_1}^{Q_2} \phi(q) \frac{\dee q}{q},
\]
we get
\begin{equation}
\label{Fpsibounds}
\phi(Q_2) \log(Q_2/Q_1) \leq F_\psi(Q_1,Q_2) \leq \phi(Q_1) \log(Q_2/Q_1).
\end{equation}
In particular, since $\phi(Q_1) \leq \phi(1) = M_\psi$ we have
\begin{equation}
\label{FpsiMpsibounds}
F_\psi(Q_1,Q_2) \leq M_\psi \log(Q_2/Q_1).
\end{equation}

\begin{lemma}
\label{lemmaNk}
Fix $\alpha,\beta > 0$ and a nice approximation function $\psi$. There exists a sequence of integers $(N_k)_1^\infty$ such that for all $k\geq 1$,
\begin{equation}
\label{Nkdef}
\beta \leq F_\psi(Q^k,Q^{k + 1}) \leq \beta + M_\psi\alpha\log(2),
\end{equation}
where
\begin{align*}
N^k &\df \prod_{j = 1}^k N_j,&
Q^k &\df (N^k)^\alpha.
\end{align*}
In fact, the sequence $(N_k)_1^\infty$ can be chosen to consist of powers of 2.
\end{lemma}
In particular, letting $\alpha = \pdim/(\dimsum)$ shows that we can choose $(N_k)_1^\infty$ so as to satisfy \eqref{Fpsisim}.
\begin{remark}
\label{remarkNkconstant}
For the case occurring in the proof of Theorem \ref{theoremHDasymp}, namely $\psi(q) = \kappa^{1/\pdim} \psi_\ast(q)$ for some $\kappa > 0$, the proof of this lemma can be accomplished by letting $(N_k)_1^\infty$ be the constant sequence whose terms are all equal to $2^{\lceil \beta/(\kappa\alpha\log(2))\rceil}$.
\end{remark}
\begin{proof}[Proof of Lemma \ref{lemmaNk}]
Given $k\geq 0$ and $N_1,\ldots,N_k$, we find $N_{k + 1}\geq 2$ such that \eqref{Nkdef} holds. Let $\ell \geq 0$ be the smallest integer for which $G_\psi(N^k,2^\ell N^k) \geq \beta$, where
\[
G_\psi(M_1,M_2) \df F_\psi(M_1^\alpha,M_2^\alpha).
\]
Such an integer exists since
\[
\lim_{\ell\to\infty} G_\psi(N^k,2^\ell N^k) = \int_{Q^k}^\infty q^{\qdim - 1} \psi^\pdim(q) \;\dee q = \infty,
\]
where the last equality holds because $\psi$ is a nice approximation function (cf. Definition \ref{definitionniceapproximation}). Since $\beta > 0$, $\ell\geq 1$. So
\begin{align*}
G_\psi(N^k,2^{\ell - 1} N^k) &\leq \beta\\
G_\psi(2^{\ell - 1} N^k,2^\ell N^k) &\leq M_\psi\alpha \log(2), \by{\eqref{FpsiMpsibounds}}
\end{align*}
and adding these inequalities yields
\[
\beta \leq G_\psi(N^k,2^\ell N^k) \leq \beta + M_\psi\alpha \log(2).
\]
Letting $N_{k + 1} = 2^\ell$ finishes the proof.
\end{proof}


\subsection{Homogeneous dynamics}
\label{subsectioncorrespondence}
Although we do not use it directly, it is worth recalling the connection between Diophantine approximation and dynamics in homogeneous spaces first discovered by S. G. Dani \cite{Dani3} and subsequently generalized by D. Y. Kleinbock and G. A. Margulis \cite{KleinbockMargulis}. To state their result precisely, we need some notation. This notation will end up being convenient for our purposes as well.

Let $\Dimsum = \dimsum$, let $\Omega_\Dimsum$ denote the set of unimodular (i.e. covolume one) lattices in $\R^\Dimsum$, and let $\Lambda_* = \Z^\Dimsum \in \Omega_\Dimsum$. Let $G = \SL_\Dimsum(\R)$ and $\Gamma = \SL_\Dimsum(\Z)$. Note that $\Omega_\Dimsum$ is isomorphic to the homogeneous space $G/\Gamma$ via the mapping $g\Gamma \mapsto g(\Lambda_*)$. Let the function $\Delta:\Omega_\Dimsum\to\R$ be defined by the formula
\[
\Delta(\Lambda) = -\log\min_{\substack{\rr\in\Lambda \\ \rr\neq\0}} \|\rr\|,
\]
where
\begin{equation}
\label{rnormdef}
\|(\pp,\qq)\| \df \|\pp\| \vee \|\qq\| = \|\pp\|_\pgreek \vee \|\qq\|_\qgreek .
\end{equation}
For each $A\in \MM$ and $t \in \R$, let
\begin{align*}
u_A &\df \left[\begin{array}{ll}
I_\pdim & -A\\
& I_\qdim
\end{array}\right],&
g_t &\df \left[\begin{array}{ll}
e^{t/\pdim}I_\pdim &\\
& e^{-t/\qdim}I_\qdim
\end{array}\right],
\end{align*}
and note that $u_A, g_t \in G$. Here $I_k$ denotes the $k$-dimensional identity matrix.

\begin{remark*}
The maps $A\mapsto u_A$ and $t\mapsto g_t$ are exponential homomorphisms (in the sense that their domains are additive groups but their common codomain is the multiplicative group $G$). Dynamically, the subgroup $(u_A)$ is a multiparameter unipotent flow on $\Omega_\Dimsum$, while the subgroup $(g_t)$ is a one-parameter diagonal flow on $\Omega_\Dimsum$.
\end{remark*}

Let $\delta = \Dimprod/\Dimsum = \dimprod/(\dimsum)$. Then the conjugation relation between the $(u_A)$ flow and the $(g_t)$ flow can be expressed as follows:
\begin{equation}
\label{deltamotivation}
g_{-\delta t} u_A g_{\delta t} = u_{e^{-t}A}.
\end{equation}
We are now ready to state the Kleinbock--Margulis correspondence principle:

\begin{theorem}[{\cite[Theorem 8.5]{KleinbockMargulis}}]
\label{theoremkleinbockmargulis}
Let $\psi:\CO 1\infty\to(0,\infty)$ be a nonincreasing continuous function, and let $r_\psi:\CO{t_0}\infty\to\R$ be defined by the formula
\begin{equation}
\label{rpsi}
\psi(e^{t/\qdim - r_\psi(t)}) = e^{-t/\pdim - r_\psi(t)}
\end{equation}
as well as the requirement that both sides should be decreasing with respect to $t$. (The existence and uniqueness of such a function is proven in \cite[Lemma 8.3]{KleinbockMargulis}). Given $A\in\MM$, we have $A\in \Approx\psi$ if and only if there exist arbitrarily large $t \geq 0$ such that
\begin{equation}
\label{kleinbockmargulis}
\Delta(g_t u_A \Lambda_*) \geq r_\psi(t).
\end{equation}
Equivalently, if for each $t \geq 0$ we let
\begin{equation}
\label{Kpsitdef}
K(\psi,t) = \{\Lambda\in \Omega_\Dimsum : \Delta(\Lambda) < r_\psi(t)\},
\end{equation}
then $A\in \Approx\psi$ if and only if there exist arbitrarily large $t \geq 0$ such that $g_t u_A \Lambda_* \notin K(\psi,t)$.
\end{theorem}
\begin{remark*}
Theorem \ref{theoremkleinbockmargulis} was originally stated for the case where $\|\cdot\|_\firstgreek$ and $\|\cdot\|_\secondgreek$ are both the max norm, but the only place this is needed in the proof is to ensure that the relation \eqref{rnormdef} holds.
\end{remark*}

\draftnewpage
\section{Lattice estimates}
\label{sectionestimates}

The goal of this section is to justify the heuristic calculation \eqref{FpsiQheur}, but in a ``local'' way that will later allow us to deduce estimates on the evaporation rates of certain trees. This ``localness'' will be represented by the fact that we consider an arbitrary unimodular lattice $\Lambda\in\Omega_\Dimsum$ rather than just the lattice $\Lambda_\ast = \Z^\Dimsum$. Specifically, let us fix:
\begin{itemize}
\item a lattice $\Lambda\in\Omega_\Dimsum$,
\item numbers $Q_1,Q_2$ with $1 \leq Q_1 \leq Q_2$,
\item a function $\psi:[Q_1,Q_2] \to (0,\infty)$ such that $\psi/\psi_\ast$ is nonincreasing.
\end{itemize}
Then the goal of this section is to prove the following theorem:
\begin{theorem}
\label{theoremestimates}
The set
\[
W_{\psi,\Lambda}(Q_1,Q_2) \df \bigcup_{\substack{\rr\in \Lambda \\ Q_1 \leq \|\qq\| \leq Q_2}} \Delta_\psi(\rr)
\]
(cf. \eqref{Deltapsidef}) satisfies
\begin{equation}
\label{estimates}
\lebesgue_\KK\big(W_{\psi,\Lambda}(Q_1,Q_2)\big)
\eqsim{\substack{Q_1/\Irr(\Lambda)\to\infty \\ \phi(Q_1)/F_\psi(Q_1,Q_2) \to 0 \\ F_\psi(Q_1,Q_2) \to 0}} \eta_\greeks F_\psi(Q_1,Q_2),
\end{equation}
where
\begin{equation}
\label{IrrLambdadef}
\Irr(\Lambda) \df \minkowski^{-(2\generic - 1)}(\Lambda)
\end{equation}
and $\phi$ is defined by \eqref{phidef}. Here $\minkowski(\Lambda)$ denotes the first Minkowski minimum of $\Lambda$ (with respect to the unit ball), i.e.
\[
\minkowski(\Lambda) \df \min_{\rr\in \Lambda\butnot\{\0\}} \|\rr\|.
\]
\end{theorem}

\begin{remark*}
The ``irregularity'' $\Irr(\Lambda)$ should be thought of as a way to measure how far away $\Lambda$ is from the ``standard lattice'' $\Lambda_\ast$. The precise formula for $\Irr(\Lambda)$ is used in only two places: the proof of Lemma \ref{lemmaregular} and the proof of Claim \ref{claimtopconvergence}.
\end{remark*}

The following corollary is what we will actually use in the proof of Theorem \ref{maintheoremv2}:

\begin{corollary}
\label{corollaryestimates}
Let $\psi:(0,\infty)\to (0,\infty)$ be a nice approximation function and let the sequence $(N_k)_1^\infty$ be as in Lemma \ref{lemmaNk}. Fix $\omega\in E^*$, and let
\begin{align*}
k &= |\omega|,&
g_\omega &= g_{\delta\log(N^k)} u_{\pi(\omega)},&
\Lambda_\omega &= g_\omega \Lambda_\ast.
\end{align*}
Then for all $Q^k \leq Q_1 \leq Q_2 \leq Q^{k + 1}$,
\[
\lebesgue_\KK\big(\KK_\omega\cap W_\psi(Q_1,Q_2)\big) \eqsim{\substack{Q_1/(Q^k \Irr(\Lambda_\omega)) \to \infty \\ \beta \to 0 \\ M_\psi/\beta \to 0 \\ F_\psi(Q^k,Q_1)/\beta \to 0 \\ F_\psi(Q_2,Q^{k +  1})/\beta \to 0}} (N^k)^{-\Dimprod} \eta_\greeks \beta.
\]
\end{corollary}
\begin{proof}
This follows from applying Theorem \ref{theoremestimates} to the lattice $\Lambda' = \Lambda_\omega$, the numbers $Q_i' = Q_i/Q^k$ ($i = 1,2$), and the function $\psi'(q) = (Q^k)^{\qdim/\pdim} \psi(Q^k q)$, and then using \eqref{Nkdef} and the formulas
\begin{align*}
\KK_\omega\cap W_{\psi,\Lambda_*}(Q_1,Q_2) &= \Phi_\omega\big(\KK \cap W_{\psi',\Lambda'}(Q_1',Q_2')\big)\\
F_{\psi'}(Q_1',Q_2') &= F_\psi(Q_1,Q_2) = F_\psi(Q^k,Q^{k + 1}) - [F_\psi(Q^k,Q_1) + F_\psi(Q_2,Q^{k + 1})]\\
\phi(Q_1) &\leq M_\psi.
\qedhere\end{align*}
\end{proof}

This corollary is the only part of this section which is used in the proof of Theorem \ref{maintheoremv2}, so Sections \ref{sectioneasydirection} and \ref{sectionharddirection} can be read independently of this section.

The proof of Theorem \ref{theoremestimates} will consist of three main steps:
\begin{itemize}
\item (Lemma \ref{lemmaestimates1}) estimating the measure of $W_{\psi,\Lambda}(Q_1,Q_2)$ ``with multiplicity'', i.e. estimating the sum of the measures of the sets $\Delta_\psi(\rr)$ ($\rr\in\Lambda$, $Q_1 \leq \|\qq\| \leq Q_2$);
\item (Lemma \ref{lemmaestimates2}) proving a ``quasi-independence on average'' result for a certain subset of $\Lambda$, the ``regular primitive'' vectors;
\item combining to estimate the measure of $W_{\psi,\Lambda}(Q_1,Q_2)$ ``without multiplicity''.
\end{itemize}

Before going into the proof of the main lemmas \ref{lemmaestimates1} and \ref{lemmaestimates2}, we need some preliminaries. For these preliminaries, we forget about the numbers $Q_1,Q_2$ and the function $\psi$, and concentrate only on the lattice $\Lambda \in \Omega_\Dimsum$.

\subsection{Preliminaries}
We collect here various facts needed in the proof of Theorem \ref{theoremestimates}.


\begin{lemma}
\label{lemmacovolumedual}
Let $\Lambda$ be a unimodular lattice in $\R^\generic$, and let
\[
\Lambda^* \df \{\ss\in\R^\generic : \rr\cdot\ss \in \Z \all \rr\in\Lambda\}
\]
be its dual lattice. Let $V\leq \R^\generic$ be a subspace such that $\Lambda\cap V$ is a lattice in $V$, and let
\[
V^\perp \df \{\ss\in\R^\generic : \rr\cdot\ss = 0 \all \rr\in V\}
\]
be its dual subspace. Then
\[
\Covol(\Lambda\cap V) \approx \Covol(\Lambda^*\cap V^\perp).
\]
\end{lemma}
Here the \emph{covolume} $\Covol(\Gamma)$ of a lattice $\Gamma$ in a normed vector space $(V,\|\cdot\|)$ is the volume of a fundamental domain\Footnote{We recall that a set $\DD\subset\R^\generic$ is called a \emph{fundamental domain} of $\Lambda$ if $\DD + \Lambda = \R^\generic$ and $\Int(\DD)\cap (\rr + \Int(\DD)) = \emptyset$ for all $\rr\in\Lambda\butnot\{\0\}$. In this paper, we make the additional assumption that $\DD\butnot\Int(\DD)$ has Lebesgue measure zero, so that $\lebesgue_{\R^\generic} = \sum_{\rr\in\Lambda} \lebesgue_{\rr + \DD}$. \label{footnotefundamental}} of $\Gamma$ with respect to $\lebesgue_V$, where $\lebesgue_V$ denotes Lebesgue measure of $V$ normalized so that the unit ball of $\|\cdot\|$ in $V$ has measure 1.\Footnote{Note that this definition disagrees slightly with the notion of covolume used in \6\ref{subsectioncorrespondence}, in which Lebesgue measure was assumed to be normalized so that $\lebesgue([0,1]^d) = 1$. Hopefully this will not cause any confusion.}
\begin{proof}
Without loss of generality suppose that the norm $\|\cdot\|$ is Euclidean, and we will show that in this case
\[
V_{\dim(V)} \Covol(\Lambda\cap V) = V_{\dim(V^\perp)} \Covol(\Lambda^*\cap V^\perp),
\]
where $V_k$ denotes the volume of the Euclidean unit ball in $\R^k$.

Indeed, let $\rr_1,\ldots,\rr_\generic$ be a basis of $\Lambda$ such that $\rr_1,\ldots,\rr_k$ is a basis of $V$, where $k = \dim(V)$. Let $\ss_1,\ldots,\ss_\generic$ denote the dual basis, so that $\ss_{k + 1},\ldots,\ss_\generic$ is a basis of $\Lambda^*\cap V^\perp$. Since $\Lambda$ is unimodular,
\[
(\rr_1\Wedge\cdots\Wedge\rr_k)\Wedge(\rr_{k + 1}\Wedge\cdots\Wedge\rr_\generic) = 1 = (\rr_{k + 1}\Wedge\cdots\Wedge\rr_\generic)\cdot(\ss_{k + 1}\Wedge\cdots\Wedge\ss_\generic)
\]
(see e.g. \cite[\63.2 and \65.5.1]{Winitzki}), and thus the equality
\[
(\rr_1\Wedge\cdots\Wedge\rr_k)\Wedge\omega = \omega\cdot(\ss_{k + 1}\Wedge\cdots\Wedge\ss_\generic) \;\;\;\;\; \left(\omega\in {\bigwedge}^{\generic - k} \R^\generic\right)
\]
can be checked by verifying that it holds on a basis of $\bigwedge^{\generic - k} \R^\generic$. So
\begin{align*}
&V_{\dim(V)} \Covol(\Lambda\cap V) = \|\rr_1\Wedge\cdots\Wedge \rr_k\| = \max_{\|\omega\| = 1} \|(\rr_1\Wedge\cdots\Wedge\rr_k)\Wedge\omega\|\\
= &\max_{\|\omega\| = 1} \|\omega\cdot(\ss_{k + 1}\Wedge\cdots\Wedge\ss_\generic)\| = \|\ss_{k + 1}\Wedge\cdots\Wedge \ss_\generic\| = V_{\dim(V^\perp)} \Covol(\Lambda^*\cap V^\perp).
\qedhere\end{align*}
\end{proof}

Another important fact about dual lattices is that if $\lambda_i$ denotes the $i$th successive Minkowski minimum, then
\begin{equation}
\label{casselsdual}
\lambda_i(\Lambda^*) \approx \lambda^{-1}_{d + 1 - i}(\Lambda);
\end{equation}
see e.g. \cite[Theorem VIII.5.VI]{Cassels3}.

\begin{notation}
\label{notation(rho)}
If $\DD\subset\R^\generic$ and $f:\R^\generic\to\Rplus$, we let
\begin{align*}
f^{(\DD)}(\rr) &= \sup_{\xx\in \rr + \DD} f(\xx),&
f_{(\DD)}(\rr) &= \inf_{\xx\in \rr + \DD} f(\xx),
\end{align*}
and if $\rho > 0$ we let
\begin{align*}
f^{(\rho)} &= f^{(\ball\0\rho)},&
f_{(\rho)} &= f_{(\ball\0\rho)}.
\end{align*}
\end{notation}

The following fundamental observation allows us to compare sums over the lattice $\Lambda$ with integrals:

\begin{observation}
\label{observationlatticecomparison}
Let $f:\R^\generic\to\Rplus$, and let $\DD$ be a fundamental domain of a unimodular lattice $\Lambda \in \Omega_\generic$ (cf. Footnote \ref{footnotefundamental}). Then
\begin{equation}
\label{latticecomparison1}
\int_{\R^\generic} f_{(\DD)}(\rr) \;\dee\rr \leq \sum_{\rr\in\Lambda} f(\rr) \leq \int_{\R^\generic} f^{(\DD)}(\rr) \;\dee\rr.
\end{equation}
Furthermore, if $\cdl$ denotes the codiameter of $\Lambda$, i.e.
\[
\cdl = \sup_{\xx\in\R^\generic} \dist(\xx,\Lambda),
\]
then
\begin{equation}
\label{latticecomparison2}
\int_{\R^\generic} f_{(\cdl)}(\rr) \;\dee\rr \leq \sum_{\rr\in\Lambda} f(\rr) \leq \int_{\R^\generic} f^{(\cdl)}(\rr) \;\dee\rr.
\end{equation}
\end{observation}
\begin{proof}
For each $\rr\in\Lambda$,
\[
\int_{\rr - \DD} f_{(\DD)}(\rr') \;\dee \rr' \leq f(\rr) \leq \int_{\rr - \DD} f^{(\DD)}(\rr') \;\dee \rr',
\]
and summing over all $\rr\in\Lambda$ demonstrates \eqref{latticecomparison1}. Next, let $\DD$ be the Dirichlet fundamental domain for $\Lambda$ centered at $\0$, i.e. the set
\begin{equation}
\label{dirichletdomain}
\DD = \{\xx\in \R^\generic : \dist(\rr,\xx) \geq \|\xx\| \all \rr\in\Lambda\}.
\end{equation}
Then $\sup_{\xx\in\DD} \|\xx\| = \cdl$, so $\DD \subset \ball\0\cdl$ and thus
\[
\int_{\rr - \DD} f_{(\cdl)}(\rr') \;\dee \rr'
\leq \int_{\rr - \DD} f_{(\DD)}(\rr') \;\dee \rr' \leq f(\rr) \leq \int_{\rr - \DD} f^{(\DD)}(\rr') \;\dee \rr',
\leq \int_{\rr - \DD} f^{(\cdl)}(\rr') \;\dee \rr',
\]
and summing over all $\rr\in\Lambda$ demonstrates \eqref{latticecomparison2}.
\end{proof}

When $f(\rr)$ depends only on $\|\rr\|$, a spherical coordinates calculation allows us to simplify the integrals of Observation \ref{observationlatticecomparison}, yielding the following upper bound:

\begin{corollary}
\label{corollaryspherical}
Fix $\cdl \leq Q_1 \leq Q_2$, and let $f:[Q_1,Q_2]\to\Rplus$ be a decreasing function. Then
\begin{align*}
\sum_{\substack{\rr\in\Lambda \\ Q_1 \leq \|\rr\| \leq Q_2}} f(\|\rr\|)
&\lessapprox Q_1^\generic f(Q_1) + \int_{Q_1}^{Q_2} f(r) r^{\generic - 1} \;\dee r.
\end{align*}
\end{corollary}
\begin{proof}
Write
\[
F(q) = \begin{cases}
f(Q_1 \vee q) & \text{if } q \leq Q_2\\
0 & \text{if } q > Q_2
\end{cases}.
\]
Then
\begin{align*}
\sum_{\substack{\rr\in\Lambda \\ Q_1 \leq \|\rr\| \leq Q_2}} f(\|\rr\|)
&\leq \int_{\R^\generic} F(\|\rr\| - \cdl) \;\dee\rr \by{\eqref{latticecomparison2}}\\
&\lessapprox \int_0^\infty r^{\generic - 1} F(r - \cdl) \;\dee r \note{Lemma \ref{lemmasphericalcoordinates}}\\
&= \int_{-\cdl}^\infty (r + \cdl)^{\generic - 1} F(r) \;\dee r\\
&= \frac{1}{\generic} (Q_1 + \cdl)^\generic f(Q_1) + \int_{Q_1}^{Q_2} (r + \cdl)^{\generic - 1} f(r) \;\dee r \noreason\\
&\approx Q_1^\generic f(Q_1) + \int_{Q_1}^{Q_2} r^{\generic - 1} f(r)\;\dee r. \since{$\cdl \leq Q_1$}
&\qedhere\end{align*}
\end{proof}

\subsection{Regular lattice vectors}
\label{subsectionirregularity}
The goal of this subsection is to find a ``large'' subset of $\Lambda$ along which quasi-independence holds. As noted in Section \ref{sectionheuristic}, we will need to restrict at least to the set of primitive vectors. But actually we may need to restrict further; some primitive vectors may still be ``bad'' in the sense that including them in our set would cause problems for quasi-independence. It turns out that the right way to measure how ``bad'' a vector is in this sense is given by the following definition:

\begin{definition}
\label{definitionirregularity}
The \emph{irregularity} of a vector $\rr\in\Lambda\butnot\{\0\}$ (with respect to $\Lambda$) is the number
\begin{equation}
\label{irrdef}
\Irr(\rr) \df \frac{\|\rr\|}{\left(\minkowski(\Lambda^*\cap\rr^\perp)\right)^{\generic - 1}}\cdot
\end{equation}
\end{definition}
As before, $\minkowski$ denotes the first Minkowski minimum and $\Lambda^*$ denotes the dual lattice. Note that by Lemma \ref{lemmacovolumedual}, we have
\[
\|\rr\| \geq \Covol(\Lambda\cap \R\rr) \eucq \Covol(\Lambda^*\cap\rr^\perp),
\]
and thus Minkowski's theorem implies that the irregularity of $\rr$ is bounded from below (by a constant depending on $\|\cdot\|$). Intuitively, the irregularity of $\rr$ tells us about the ``worst hyperplane'' that $\rr$ is contained in, where ``worst'' means ``smallest covolume''. The advantage of vectors with low irregularity is that they can be used to construct fundamental domains of $\Lambda$ which are approximately cylindrical; cf. the proof of Lemma \ref{lemmaestimates2} below. The following lemma shows that there is a sense in which the irregularity is bounded from above ``on average'':

\begin{lemma}[Almost all lattice vectors are regular]
\label{lemmaregular}
If $\generic \geq 2$, then
\[
\lim_{K\to\infty} \sup_{\Lambda\in\Omega_\generic} \sup_{Q \geq K\Irr(\Lambda)} \frac{\#\{\rr \in \Lambda : \Irr(\rr) \geq K, \|\rr\|\leq Q\}}{Q^\generic} = 0,
\]
where $\Irr(\Lambda) \df \minkowski^{-(2\generic - 1)}(\Lambda)$ as in \eqref{IrrLambdadef}.
\end{lemma}

\begin{proof}
Fix $K \geq 1$, $\Lambda\in\Omega_\generic$, and $Q\geq K\Irr(\Lambda)$. Fix $\rr\in\Lambda$ with $\Irr(\rr) \geq K$ and $\|\rr\|\leq Q$. Then
\[
\minkowski(\Lambda^*\cap \rr^\perp) = \left(\frac{\|\rr\|}{\Irr(\rr)}\right)^{1/(\generic - 1)} \leq S \df \left(\frac QK\right)^{1/(\generic - 1)},
\]
so there exists $\ss\in\Lambda^*_\prim\cap\rr^\perp$ such that $\|\ss\|\leq S$, where $\Lambda^*_\prim$ denotes the set of primitive vectors of $\Lambda^*$. Equivalently,
\[
\rr\in \bigcup_{\substack{\ss\in\Lambda^*_\prim \\ \|\ss\|\leq S}} \Lambda\cap\ss^\perp.
\]
It follows that
\begin{equation}
\label{regular1}
\#\{\rr \in \Lambda : \Irr(\rr) \geq K, \; \|\rr\|\leq Q\}
\leq \sum_{\substack{\ss\in\Lambda^*_\prim \\ \|\ss\|\leq S}} \#(\Lambda\cap\ss^\perp\cap \ball\0 Q).
\end{equation}
Next, fix $\ss\in\Lambda^*_\prim$ with $\|\ss\|\leq S$, and let $\Gamma = \Lambda\cap\ss^\perp$. Then
\begin{equation}
\label{Gammaball0Q}
\begin{split}
\#(\Gamma\cap \ball\0 Q)\cdot \Covol(\Gamma)
&= \lebesgue_{\ss^\perp}\left(\bigcup_{\rr\in \Gamma\cap \ball\0 Q} \rr+\DD_\Gamma\right)
\leq \lebesgue_{\ss^\perp}\Big(B(\0,Q + \Codiam{\Gamma})\Big) = \left(Q + \Codiam{\Gamma}\right)^{\generic - 1},
\end{split}
\end{equation}
where $\DD_\Gamma$ denotes the Dirichlet domain of $\Gamma$ centered at $0$, and $\lebesgue_{\ss^\perp}$ denotes Lebesgue measure of $\ss^\perp$, normalized so that $\lebesgue_{\ss^\perp}(\ss^\perp\cap B(\0,1)) = 1$. Now by Lemma \ref{lemmacovolumedual} we have
\begin{equation}
\label{Gammas}
\Covol(\Gamma) \eucq \|\ss\|
\end{equation}
and thus by Minkowski's Second Theorem \cite[VIII.4.V]{Cassels3},
\begin{equation}
\label{CodiamGammaQ}
\begin{split}
\Codiam{\Gamma}
\approx \lambda_{d - 1}(\Gamma)
&\lessapprox \frac{\Covol(\Gamma)}{\minkowski^{\generic - 2}(\Gamma)}
\leq \frac{\Covol(\Gamma)}{\minkowski^{\generic - 2}(\Lambda)}
\eucq \frac{\|\ss\|}{\Irr^{-(\generic - 2)/(2\generic - 1)}(\Lambda)}\\
&\leq \frac{(Q/K)^{1/(\generic - 1)}}{(Q/K)^{-(\generic - 2)/(2\generic - 1)}}
\leq Q/K
\leq Q.
\end{split}
\end{equation}
Plugging in \eqref{Gammas} and \eqref{CodiamGammaQ} into \eqref{Gammaball0Q} gives
\[
\#(\Gamma\cap \ball\0 Q) \lessapprox \frac{Q^{\generic - 1}}{\|\ss\|}\cdot
\]
Plugging this into \eqref{regular1}, we get
\begin{align*}
&Q^{-\generic} \#\{\rr \in \Lambda : \Irr(\rr) \geq K, \|\rr\|\leq Q\} \noreason\\
&\lessapprox \sum_{\substack{\ss\in\Lambda^*_\prim \\ \|\ss\|\leq S}} 
\frac{1}{Q \|\ss\|}\\
&\leq \frac{\#\{\ss\in\Lambda^* : \|\ss\| \leq \Codiam{\Lambda^*}\}}{Q\minkowski(\Lambda^*)} + \sum_{\substack{\ss\in\Lambda^* \\ \Codiam{\Lambda^*} \leq \|\ss\|\leq S}} \frac{1}{Q \|\ss\|} \noreason\\
&\lessapprox \frac{\Codiam{\Lambda^*}^\generic}{Q\minkowski(\Lambda^*)} + \int_0^{S} \frac{r^{\generic - 2}}{Q}\;\dee r \by{Corollary \ref{corollaryspherical}}\\
&\lessapprox \left.\frac{\Codiam{\Lambda^*}^{2\generic - 1}}{Q} + \frac{r^{\generic - 1}}{Q}\right\vert_{r = 0}^{S} \since{$\minkowski(\Lambda^*) \gtrapprox \Codiam{\Lambda^*}^{-(\generic - 1)}$ and $\generic \geq 2$}\\
&\approx \frac{\Irr(\Lambda)}{Q} + \frac{1}{K} \leq \frac 2K\cdot \since{$\Codiam{\Lambda^*} \approx \minkowski^{-1}(\Lambda)$ and $Q\geq K\Irr(\Lambda)$}
\end{align*}
Taking the appropriate suprema and letting $K \to \infty$ completes the proof. Note that the asymptotic inequality $\minkowski(\Lambda^*) \gtrapprox \Codiam{\Lambda^*}^{-(\generic - 1)}$ follows from combining Minkowski's second theorem on successive minima with the asymptotic $\Codiam{\Lambda^*} \approx \lambda_d(\Lambda^*)$, while the asymptotic $\Codiam{\Lambda^*} \approx \minkowski^{-1}(\Lambda)$ follows from combining \eqref{casselsdual} with the same asymptotic.
\end{proof}

For each $K \geq 1$, write
\begin{equation}
\label{epsilonKdef}
\epsilon_K = \sup_{\Lambda\in\Omega_\generic} \sup_{Q \geq K\Irr(\Lambda)} \frac{\#\{\rr \in \Lambda : \Irr(\rr) \geq K, \|\rr\|\leq Q\}}{Q^\generic}\cdot
\end{equation}
Then Lemma \ref{lemmaregular} says that if $\generic \geq 2$, then $\epsilon_K \to 0$ as $K \to \infty$. The notation $\epsilon_K$ allows us to write down the following analogue of Corollary \ref{corollaryspherical} for sums over lattice vectors with large irrationality:

\begin{observation}
\label{observationepsilonK}
Fix $K\geq 1$ and $\Lambda\in\Omega_\generic$, and let $f:\Rplus\to\Rplus$ be a decreasing function which is constant on $[0,K\Irr(\Lambda)]$. Then
\[
\sum_{\substack{\rr\in \Lambda \\ \Irr(\rr) \geq K}} f(\|\rr\|) \leq \generic \epsilon_K \int_0^\infty r^{\generic - 1} f(r) \;\dee r.
\]
\end{observation}
\begin{proof}
It suffices to check this for functions of the form $f(q) = \big[q \leq Q\big]$ ($Q \geq K \Irr(\Lambda)$), and for these functions the inequality follows from the definition of $\epsilon_K$.
\end{proof}


\subsection{Proof of Theorem \ref{theoremestimates}}
\label{subsectionestimates}
In the remainder of this section, we fix $Q_1,Q_2,\psi$ as in the start of this section.

The following lemma allows us to simplify the notation somewhat:

\begin{lemma}[Good lattice vectors satisfy $\|(\pp,\qq)\| = \|\qq\|$]
\label{lemmaqr}
Suppose that $\psi(Q_1) \leq 1/2$ and $\diamcube \leq 1/2$. Then for all $(\pp,\qq)\in \R^\Dimsum$ with $Q_1 \leq \|\qq\| \leq Q_2$ such that $\KK\cap\Delta_\psi(\pp,\qq)\neq\emptyset$, we have
\begin{equation}
\label{qr}
\|(\pp,\qq)\| = \|\qq\|.
\end{equation}
\end{lemma}
\begin{proof}
If $A\in \KK\cap\Delta_\psi(\pp,\qq)$, then
\[
\|\pp\| \leq \|A\qq - \pp\| + \|A\|\cdot\|\qq\| \leq \psi(\|\qq\|) + \diamcube\cdot\|\qq\|.
\]
Since $\|\qq\|\geq Q_1 \geq 1$ and $\psi$ is nonincreasing, we have
\[
\psi(\|\qq\|) \leq \psi(Q_1) \leq \psi(Q_1)\|\qq\|.
\]
Thus $\|\pp\| \leq (\psi(Q_1) + \diamcube)\|\qq\| \leq \|\qq\|$, which implies \eqref{qr}.
\end{proof}

To make things easier, in this section we make the standing assumption that the hypotheses of Lemma \ref{lemmaqr} are satisfied:
\begin{assumption}
\label{assumption12}
$\psi(Q_1) \leq 1/2$ and $\diamcube \leq 1/2$.
\end{assumption}

In fact, Assumption \ref{assumption12} may be made without loss of generality: first of all, by replacing $\|\cdot\|_\pgreek$ by $\multerror \|\cdot\|_\pgreek$ for an appropriate $\multerror > 0$, we can assume that $\diamcube \leq 1/2$; second of all, since in Theorem \ref{theoremestimates} we are trying to prove an asymptotic which includes the convergence assumption $\phi(Q_1) \to 0$, it follows that we can assume $\phi(Q_1)$ is as small as we want and in particular that $\psi(Q_1) \leq 1/2$.

For convenience of notation we introduce the function
\[
\w\psi(q) = \psi(Q_1\vee q)
\]
and we note that $\w\psi$ is still nonincreasing (although the function $\w\phi(q) = q^\qdim \w\psi^\pdim(q)$ is not nonincreasing). We also let $\w\Psi(q) = \Psi(Q_1\vee q)$.

In the following lemma, $\Lrep$ denotes a subset of $\Lambda_\prim$ (the set of primitive vectors of $\Lambda$) such that for all $\rr\in\Lambda_\prim$, we have $\#(\{\rr,-\rr\}\cap\Lrep) = 1$.

\begin{lemma}[Estimates of the measure of $W_{\psi,\Lambda}(Q_1,Q_2)$, with multiplicity]
\label{lemmaestimates1}
We have
\begin{equation}
\label{estimates11}
\sum_{\substack{\rr\in\Lrep \\ Q_1 \leq \|\qq\| \leq Q_2}}  \lebesgue_\KK\big(\Delta_\psi(\rr)\big)
\eqsim{\substack{Q_1/\Irr(\Lambda) \to \infty \\ \phi(Q_1)/F_\psi(Q_1,Q_2) \to 0}}
\eta_\greeks F_\psi(Q_1,Q_2).
\end{equation}
Moreover, for all $K\geq 1$ such that $Q_1 \geq K\Irr(\Lambda)$ we have
\begin{equation}
\label{estimates12}
\sum_{\substack{\rr\in\Lambda \\ Q_1 \leq \|\qq\| \leq Q_2 \\ \Irr(\rr) \geq K}}  \lebesgue_\KK\big(\Delta_\psi(\rr)\big)
\LessLess{\substack{K \to \infty \\ \phi(Q_1)/F_\psi(Q_1,Q_2) \to 0}}
F_\psi(Q_1,Q_2)
\end{equation}
and thus
\begin{equation}
\label{estimates13}
\sum_{\substack{\rr\in\Lrep \\ Q_1 \leq \|\qq\| \leq Q_2 \\ \Irr(\rr) \leq K}}  \lebesgue_\KK\big(\Delta_\psi(\rr)\big)
\eqsim{\substack{K \to \infty \\ \phi(Q_1)/F_\psi(Q_1,Q_2) \to 0}}
\eta_\greeks F_\psi(Q_1,Q_2).
\end{equation}
\end{lemma}
\ignore{
\begin{proof}[Heuristic calculation]
On one hand
\begin{align*}
\sum_{\substack{\rr\in\Lambda \\ Q_1 \leq \|\qq\| \leq Q_2}}  \lebesgue_\KK\big(\Delta_\psi(\rr)\big)
&\sim \int_{Q_1 \leq \|\qq\| \leq Q_2} \lebesgue_\KK\big(\Delta_\psi(\rr)\big) \;\dee\rr\\
&= \int_{Q_1 \leq \|\qq\| \leq Q_2} \int_\KK \lebesgue_{\R^\pdim}(\{\pp\in\R^\pdim : \|A\qq - \pp\| \leq \psi(\|\qq\|)\}) \;\dee A \;\dee\qq \noreason\\
&= \int_{Q_1 \leq \|\qq\| \leq Q_2} V_\pgreek \psi^\pdim(\|\qq\|) \;\dee\qq\\
&= V_\pgreek \qdim V_\qgreek \int_{Q_1}^{Q_2} q^{\qdim - 1} \psi^\pdim(q) \;\dee q \note{Lemma \ref{lemmasphericalcoordinates}}\\
&= \qdim \Vgreeks F_\psi(Q_1,Q_2)
\end{align*}
and on the other hand
\begin{align*}
\sum_{\substack{\rr\in\Lambda \\ Q_1 \leq \|\qq\| \leq Q_2}} \lebesgue_\KK\big(\Delta_\psi(\rr)\big)
&= 2\sum_{k = 1}^\infty \sum_{\substack{\rr\in k\Lrep \\ Q_1 \leq \|\qq\| \leq Q_2}} \lebesgue_\KK\big(\Delta_\psi(\rr)\big)\\
&\sim 2\sum_{k = 1}^\infty k^{-\Dimsum} \sum_{\substack{\rr\in \Lrep \\ Q_1 \leq \|\qq\| \leq Q_2}} \lebesgue_\KK\big(\Delta_\psi(\rr)\big)\\
&= 2\zeta(\Dimsum) \sum_{\substack{\rr\in \Lrep \\ Q_1 \leq \|\qq\| \leq Q_2}} \lebesgue_\KK\big(\Delta_\psi(\rr)\big).
\end{align*}
\end{proof}
To make this heuristic argument rigorous, we need to justify the two parts of the calculation labeled $\sim$.
}
\begin{proof}
\hspace{0in}~

\textbf{Step 1:} Initial calculation. Consider the function $\eta:\R^\Dimsum\butnot(\R^\pdim\times\{\0\})\to\Rplus$ defined by the formula
\[
\eta(\pp,\qq) \df \prod_{i = 1}^\pdim \left(\bigast_{\substack{j = 1 \\ q_j \neq 0}}^\qdim f_{q_j}\right)(p_i),
\]
where
\[
f_q(p) \df \begin{cases}
\frac 1q \big[p\in [0,q]\big] & \text{if } q > 0\\
\frac 1{|q|} \big[p \in [q,0]\big] & \text{if } q < 0
\end{cases} = \frac{\dee [t\mapsto qt]_* \lebesgue_{[0,1]}}{\dee\lebesgue_\R}
\]
and $\bigast$ denotes convolution. Here $f_* \mu$ denotes the pushforward of a measure $\mu$ under a function $f$, and $\frac{\dee\mu}{\dee\nu}$ denotes the Radon--Nikodym derivative of $\mu$ with respect to $\nu$. The significance of $\eta$ is that for each $\qq\in\R^\qdim\butnot\{\0\}$, $\eta(\cdot,\qq)$ is a nonnegative Riemann integrable function which satisfies
\begin{equation}
\label{etaS}
\int_S \eta(\pp,\qq) \;\dee\pp = \lebesgue_\KK(\{A\in \KK : A\qq \in S\}) \all S \subset \R^\pdim.
\end{equation}
Indeed, we have
\begin{align*}
\eta \lebesgue_{\R^\pdim} &= \prod_{i = 1}^\pdim \left(\bigast_{\substack{j = 1 \\ q_j \neq 0}}^\qdim f_{q_j}\right)\lebesgue_\R\\
&= \prod_{i = 1}^\pdim \Big[(x_1,\ldots,x_\qdim)\mapsto \Sigma_1^\qdim x_j\Big]_* \prod_{j = 1}^\qdim \begin{cases}
f_{q_j} \lebesgue_\R & q_j \neq 0\\
\delta_0 & q_j = 0
\end{cases}\\
&= \prod_{i = 1}^\pdim \Big[(x_1,\ldots,x_\qdim)\mapsto \Sigma_1^\qdim x_j\Big]_* \prod_{j = 1}^\qdim
\big[t\mapsto q_j t\big]_* \lebesgue_{[0,1]}\\
&= \prod_{i = 1}^m \Big[(t_1,\ldots,t_\qdim)\mapsto \Sigma_1^\qdim q_j t_j\Big]_*\lebesgue_{[0,1]^\qdim}\\
&= \Big[(t_{11},\ldots,t_{\pdim\qdim})\mapsto (\Sigma_1^\qdim t_{1 j} q_j,\ldots,\Sigma_1^\qdim t_{\pdim j} q_j)\Big]_* \lebesgue_{[0,1]^{\dimprod}}\\
&= \Big[A \mapsto A\qq\Big]_* \lebesgue_\KK.
\end{align*}
Also, note that
\begin{equation}
\label{etascaling}
\eta(\multerror\rr) = \multerror^{-\pdim} \eta(\rr) \all \multerror > 0 \all \rr\in \R^\Dimsum\butnot(\R^\pdim\times\{\0\}).
\end{equation}
Now for $\rho > 0$ let
\begin{align*}
\tau^{(\rho)} &= \max_{\|\qq\| = 1} \int_{\R^\pdim} \eta^{(\rho)}(\pp,\qq) \;\dee \pp,&
\tau_{(\rho)} &= \min_{\|\qq\| = 1} \int_{\R^\pdim} \eta_{(\rho)}(\pp,\qq) \;\dee \pp,
\end{align*}
where $\eta^{(\rho)}$, $\eta_{(\rho)}$ are as in Notation \ref{notation(rho)}. Note that $\tau^{(\rho)}\searrow 1$ and $\tau_{(\rho)} \nearrow 1$ as $\rho\to 0$. Also, by \eqref{etascaling} we have
\begin{equation}
\label{tauscaling}
\begin{split}
\int_{\R^\pdim} \eta_{(\rho)}(\pp,\qq) \;\dee \pp
&\geq \tau_{(\rho/\|\qq\|)}\\
\int_{\R^\pdim} \eta^{(\rho)}(\pp,\qq) \;\dee \pp
&\leq \tau^{(\rho/\|\qq\|)}
\end{split} \all \rho > 0 \all \qq \in \R^\qdim\butnot\{\0\}.
\end{equation}
Next, for all $\rr\in\Lambda_*$, by \eqref{etaS} we have
\begin{equation}
\label{etaintegral}
\lebesgue_\KK\big(\Delta_\psi(\rr)\big)
= \lebesgue_\KK(\{A\in\KK : A\qq \in \ball{\pp}{\psi(\|\qq\|)}\})
= \int_{\ball{\pp}{\psi(\|\qq\|)}} \eta(\pp',\qq) \;\dee\pp'
\end{equation}
and thus
\begin{align*}
&\sum_{\substack{\rr\in \Lambda\\ Q_1 \leq \|\qq\| \leq Q_2}} \lebesgue_\KK\big(\Delta_\psi(\rr)\big) \\
&\leq \sum_{\substack{\rr\in \Lambda \\ Q_1 \leq \|\qq\| \leq Q_2}} \eta^{(\psi(Q_1))}(\rr) V_\pgreek \psi^\pdim(\|\qq\|) \\
&\leq V_\pgreek \int_{\R^\Dimsum} \eta^{(\psi(Q_1) + \cdl)}(\rr) \w\psi^\pdim(\|\qq\| - \cdl) \big[Q_1 - \cdl \leq \|\qq\| \leq Q_2 + \cdl\big]\;\dee\rr  \\ \by{\eqref{latticecomparison2}}\\
&\leq V_\pgreek \tau^{\left(\frac{\psi(Q_1) + \cdl}{Q_1 - \cdl}\right)} \int_{\R^\qdim} \w\psi^\pdim(\|\qq\| - \cdl) \big[Q_1 - \cdl \leq \|\qq\| \leq Q_2 + \cdl\big]\;\dee\qq  \\ \by{\eqref{tauscaling}}\\
&= V_\pgreek \tau^{\left(\frac{\psi(Q_1) + \cdl}{Q_1 - \cdl}\right)} \qdim V_\qgreek \int_{Q_1 - \cdl}^{Q_2 + \cdl} \w\psi^\pdim(q - \cdl) q^{\qdim - 1} \;\dee q \\ \note{Lemma \ref{lemmasphericalcoordinates}}\\
&\leq \qdim \Vgreeks \tau^{\left(\frac{\psi(Q_1) + \cdl}{Q_1 - \cdl}\right)} \left(\frac{Q_1 - \cdl}{Q_1 - 2\cdl}\right)^{\qdim - 1} \int_{Q_1 - \cdl}^{Q_2 + \cdl} \w\psi^\pdim(q - \cdl) (q - \cdl)^{\qdim - 1} \;\dee q \\
&\leq \qdim \Vgreeks \tau^{\left(\frac{\psi(Q_1) + \cdl}{Q_1 - \cdl}\right)} \left(\frac{Q_1 - \cdl}{Q_1 - 2\cdl}\right)^{\qdim - 1} \big[F_\psi(Q_1,Q_2) + \phi(Q_1) \big].
\end{align*}
A similar argument (using \eqref{Fpsibounds}) shows that
\begin{align*}
&\sum_{\substack{\rr\in \Lambda\\ Q_1 \leq \|\qq\| \leq Q_2}} \lebesgue_\KK\big(\Delta_\psi(\rr)\big)\\
&\geq \qdim \Vgreeks \tau_{\left(\frac{\psi(Q_1) + \cdl}{Q_1 + \cdl}\right)} \left(\frac{Q_1 + \cdl}{Q_1 + 2\cdl}\right)^{\qdim - 1} \left[F_\psi(Q_1,Q_2) - \phi(Q_1) \log\left(\frac{Q_1 + 2\cdl}{Q_1}\right)\right].
\end{align*}
Thus
\begin{equation}
\label{beforezetacorrection}
\sum_{\substack{\rr\in \Lambda\\ Q_1 \leq \|\qq\| \leq Q_2}} \lebesgue_\KK\big(\Delta_\psi(\rr)\big)
\eqsim{\substack{Q_1/\cdl \to \infty \\ \phi(Q_1)/F_\psi(Q_1,Q_2) \to 0}}
\qdim \Vgreeks F_\psi(Q_1,Q_2).
\end{equation}

To get \eqref{estimates11}, we need to change \eqref{beforezetacorrection} to a formula about primitive vectors. Before doing this, we first demonstrate \eqref{estimates12}.

\textbf{Step 2: Ignoring irregular lattice vectors.}
Fix $\rr\in\R^\Dimsum$. Since $\eta(\pp',\qq) \lessapprox \|\qq\|^{-\pdim}$ for all $\pp'\in\R^\pdim$, \eqref{etaintegral} implies that
\begin{equation}
\label{crudebound}
\lebesgue_\KK\big(\Delta_\psi(\rr)\big) \lessapprox \|\rr\|^{-\pdim} \psi^\pdim(\|\rr\|),
\end{equation}
where we have used Lemma \ref{lemmaqr} to get $\|\qq\| = \|\rr\|$ in the case where the left-hand side is nonzero. Now fix $K \geq 1$ such that $Q_1 \geq K \Irr(\Lambda)$, and let $\epsilon_K$ be as in \eqref{epsilonKdef}. Then
\begin{align*}
\sum_{\substack{\rr\in\Lambda \\ Q_1 \leq \|\rr\| \leq Q_2 \\ \Irr(\rr) \geq K}} \lebesgue_\KK\big(\Delta_\psi(\rr)\big)
&\lessapprox \sum_{\substack{\rr\in\Lambda \\ Q_1 \leq \|\rr\| \leq Q_2 \\ \Irr(\rr) \geq K}} \Psi^\pdim(\|\rr\|) \by{\eqref{crudebound}}\\
&\leq \sum_{\substack{\rr\in\Lambda \\ \|\rr\| \leq Q_2 \\ \Irr(\rr) \geq K}} \w\Psi^\pdim(\|\rr\|)\\
&\leq \epsilon_K \Dimsum \int_0^{Q_2} q^{\generic - 1} \w\Psi^\pdim(q) \;\dee q \note{Observation \ref{observationepsilonK}}\\
&= \epsilon_K [\Dimsum F_\psi(Q_1,Q_2) + \phi(Q_1)]\\
&\hspace{-.47 in}\LessLess{\substack{K \to \infty \\ \phi(Q_1)/F_\psi(Q_1,Q_2) \to 0}} F_\psi(Q_1,Q_2),
\end{align*}
which completes the proof of \eqref{estimates12}.

\textbf{Step 3: Accounting for repeated entries.}
Fix $k\in\N$ such that $k \leq Q_1$. The argument of Step 1 goes through just as well for the function $\psi_k:[Q_1/k,Q_2/k]\to\Rplus$ defined by the formula
\[
\psi_k(q) = \psi(k q)/k,
\]
which satisfies
\begin{align*}
\Delta_{\psi'}(\rr) &= \Delta_\psi(k \rr) \all \rr\in\Lambda,&
\phi_k(Q_1/k) &= k^{-\Dimsum} \phi(Q_1),\\
F_{\psi_k}(Q_1/k,Q_2/k) &= k^{-\Dimsum} F_\psi(Q_1,Q_2).\noreason
\end{align*}
Thus, replacing $\psi$ by $\psi_k$ in \eqref{beforezetacorrection} gives
\[
\sum_{\substack{\rr\in k\Lambda\\ Q_1 \leq \|\qq\| \leq Q_2}} \lebesgue_\KK\big(\Delta_\psi(\rr)\big)
\eqsim{\substack{Q_1/(k\cdl) \to \infty \\ \phi(Q_1)/F_\psi(Q_1,Q_2) \to 0}}
\qdim \Vgreeks k^{-\Dimsum} F_\psi(Q_1,Q_2).
\]
Now let
\begin{align*}
A_k &\df \frac{1}{\qdim \Vgreeks F_\psi(Q_1,Q_2)} \sum_{\substack{\rr\in k \Lambda_\prim \\ Q_1 \leq \|\qq\| \leq Q_2}} \lebesgue_\KK\big(\Delta_\psi(\rr)\big),\\
B_k &\df \frac{1}{\qdim \Vgreeks F_\psi(Q_1,Q_2)} \sum_{\substack{\rr\in k \Lambda \\ Q_1 \leq \|\qq\| \leq Q_2}} \lebesgue_\KK\big(\Delta_\psi(\rr)\big) = \sum_{\ell\in\N} A_{k\ell},
\end{align*}
and note that
\begin{equation}
\label{beforezetacorrection2}
B_k \Plussim{\substack{Q_1/\cdl \to \infty \; (k) \\ \phi(Q_1)/F_\psi(Q_1,Q_2) \to 0}} k^{-\Dimsum}.
\end{equation}
(To switch from a multiplicative asymptotic to an additive one, we have used the fact that $k^{-\Dimsum} \leq 1$.) Recall that for all $\rr\in\Lambda$, $\Irr(\rr)$ is bounded below by a constant depending only on $\constants$, say $\Irr(\rr) \geq \delta_0 \all \rr\in\Lambda$. Now fix $K\in\N$ such that $Q_1 \geq (K/\delta_0) \Irr(\Lambda)$. Since $\Irr(k\rr) = k\Irr(\rr) \geq k \delta_0$ for all $\rr\in\Lambda_\prim$ and $k\in\N$, by \eqref{estimates12} we have
\begin{equation}
\label{totalAkbound}
\sum_{|k|\geq K} A_k \lessapprox \frac{1}{F_\psi(Q_1,Q_2)} \sum_{\substack{\rr\in\Lambda \\ Q_1 \leq \|\qq\| \leq Q_2 \\ \Irr(\rr) \geq K/\delta_0}} \lebesgue_\KK\big(\Delta_\psi(\rr)\big) \tendsto{\substack{K\to\infty \\ \phi(Q_1)/F_\psi(Q_1,Q_2) \to 0}} 0.
\end{equation}
Now let $P$ be the set of all primes strictly less than $K$, and let $S$ be the set of all square-free integers whose prime factors are all in $P$ (or equivalently, the set of products of distinct elements of $P$). Let $\mu$ denote the M\"obius function.\Footnote{For convenience, we recall that this function is given by the formula $\mu(p_1 \cdots p_k) = (-1)^k$ if $p_1,\ldots,p_k$ are distinct primes, and $\mu(n) = 0$ if $n$ is not square-free.} Then
\begin{eqnarray*}
A_1 &\Plussim{\substack{K\to \infty \\ \phi(Q_1)/F_\psi(Q_1,Q_2) \to 0}}& \sum_{\substack{k\in\N \\ k\notin \ell\N \all \ell\in P}} A_k \;\;\;\;\;\;\;\;\;\;\;\; \text{(by \eqref{totalAkbound})}\\
&=& \sum_{k\in S} \mu(k) B_k\\
&\Plussim{\substack{Q_1/\cdl \to \infty \; (K) \\ \phi(Q_1)/F_\psi(Q_1,Q_2) \to 0}}& \sum_{k\in S} \mu(k) k^{-\Dimsum} \;\;\;\;\;\;\;\;\;\;\;\; \text{(by \eqref{beforezetacorrection2})}\\
&&\text{(since $S$ is finite and depends only on $K$)}\\
&=& \prod_{k\in P}\frac{1}{1 - k^{-\Dimsum}}\\
&\eqsim{K\to\infty}& \frac{1}{\zeta(\Dimsum)}\cdot
\end{eqnarray*}
Standard quantifier logic then gives
\[
A_1 \eqsim{\substack{Q_1/\Irr(\Lambda)\to\infty \\ \phi(Q_1)/F_\psi(Q_1,Q_2) \to 0}} \frac{1}{\zeta(\Dimsum)}
\]
and rearranging gives \eqref{estimates11}. Finally, combining \eqref{estimates11} and \eqref{estimates12} and using the assumption $Q_1 \geq K\Irr(\Lambda)$ to get $Q_1/\Irr(\Lambda) \tendsto{K\to\infty} \infty$ gives \eqref{estimates13}.
\end{proof}

\begin{lemma}[Quasi-independence on average]
\label{lemmaestimates2}
Suppose that $Q_1 \geq \cdl$. Then for all $\rr = (\pp,\qq)\in\Lambda_\prim$ such that $Q_1 \leq \|\qq\| \leq Q_2$, we have
\begin{equation}
\label{quasiindep}
\sum_{\substack{\rr'\in\Lambda\butnot\Z\rr \\ \|\qq\| \leq \|\qq'\| \leq Q_2}}  \lebesgue_\KK\big(\Delta_\psi(\rr)\cap\Delta_\psi(\rr')\big) \lessapprox_{\Irr(\rr)} \Psi^\pdim(\|\rr\|) \wbar F_\psi(Q_1,Q_2),
\end{equation}
where
\[
\wbar F_\psi(Q_1,Q_2) \df F_\psi(Q_1,Q_2) + \phi(Q_1).
\]
\end{lemma}
\noindent {\it Proof.}
We can assume that $\KK\cap\Delta_\psi(\rr)\neq\emptyset$, as otherwise \eqref{quasiindep} holds trivially. Thus by Lemma \ref{lemmaqr}, we get $\|\qq\| = \|\rr\|$. Also, since $F_\psi(\|\qq\|,Q_2) + \phi(\|\qq\|) \leq F_\psi(Q_1,Q_2) + \phi(Q_1)$, we may without loss of generality assume that $Q_1 = \|\qq\|$.

Let $\pi:\R^\Dimsum\to \rr^\perp$ be the orthogonal projection map, so that $\pi(\rr) = 0$. Then $\pi(\Lambda)$ is a lattice in $\rr^\perp$. Let $\DD_1\subset\rr^\perp$ be the Dirichlet fundamental domain for $\pi(\Lambda)$ centered at $\0$ (cf. \eqref{dirichletdomain}). Let $\DD_2 = [-1/2,1/2]\rr$, and let $\DD = \DD_1 + \DD_2$. Then $\DD$ is a fundamental domain for $\Lambda$ (though not necessarily a Dirichlet domain). Note that
\begin{equation}
\label{diamDbound}
\cdd \df \max_{\xx\in\DD}\|\xx\| \leq \Codiam{\pi(\Lambda)} + \|\rr\| \leq \Codiam{\Lambda} + \|\rr\| \leq \|\qq\| + \|\rr\| = 2 Q_1.
\end{equation}

\begin{claim}
\label{claimfi}
For all $\rr'\in \Lambda\butnot\Z\rr$ and $\rr'' \in \rr' + \DD$, we have
\[
\dist(\rr',\R\rr) \gtrapprox_{\Irr(\rr)} \dist(\rr'',\R\rr).
\]
\end{claim}
\begin{subproof}
Fix $\rr'\in \Lambda\butnot\Z\rr$ and $\rr'' \in \rr' + \DD$. Then $\pi(\rr') \in \pi(\Lambda)\butnot\{\0\}$ and $\pi(\rr'') \in \rr' + \DD_1$. In particular, we have
\[
\dist(\rr',\R\rr) = \|\pi(\rr')\| \geq \minkowski(\pi(\Lambda))
\]
and thus
\[
\dist(\rr'',\R\rr) \leq \dist(\rr',\R\rr) + \Codiam{\pi(\Lambda)} \leq \left(1 + \frac{\Codiam{\pi(\Lambda)}}{\minkowski(\pi(\Lambda))}\right)\dist(\rr',\R\rr).
\]
But since $\pi(\Lambda)^* = \Lambda^*\cap\rr^\perp$, we have $\Codiam{\pi(\Lambda)} \approx \minkowski^{-1}(\Lambda^*\cap\rr^\perp)$ and $\minkowski(\pi(\Lambda)) \approx \Codiam{\Lambda^*\cap\rr^\perp}^{-1} \approx \lambda_{d - 1}^{-1}(\Lambda^*\cap\rr^\perp)$. Thus by Minkowski's Second Theorem, the ratio $\Codiam{\pi(\Lambda)}/\minkowski(\pi(\Lambda))$ is bounded from above by a constant depending only on $\Irr(\rr)$. This completes the proof.
\end{subproof}

Denote the implied constant by $k_1$ (so that $k_1$ depends on $\Irr(\rr)$). We now split the proof into two cases according to whether or not $\qdim = 1$:

\begin{proof}[Completion of the proof if $\qdim = 1$]
Without loss of generality suppose that $q > 0$ (where $\rr = (\pp,q)$) and $\|\cdot\|_\qgreek = |\cdot|$, so that $\|q\| = q = Q_1$. Now fix $\rr' = (\pp',q')\in\Lambda\butnot\Z\rr$ such that $\Delta_\psi(\rr)\cap\Delta_\psi(\rr') \neq \emptyset$ and $Q_1 \leq q' \leq Q_2$. Since
\begin{equation}
\label{Deltapsir1dim}
\Delta_\psi(\rr) = B(\pp/q,\Psi(q)), \;\;\; \Delta_\psi(\rr') = B(\pp'/q',\Psi(q')),
\end{equation}
we get
\[
\left\|\frac{\pp'}{q'} - \frac{\pp}{q} \right\| \leq \Psi(q) + \Psi(q') \leq 2\Psi(q)
\]
and thus
\[
\dist(\rr',\R\rr) \leq \dist\big((\pp',q'),(q'\pp/q,q')\big) \leq 2 q' \Psi(q).
\]
Thus by \eqref{crudebound},
\begin{align*}
&\sum_{\substack{\rr'\in\Lambda\butnot\Z\rr \\ Q_1 \leq q' \leq Q_2}}  \lebesgue_\KK\big(\Delta_\psi(\rr)\cap\Delta_\psi(\rr')\big)
\lessapprox \sum_{\substack{\rr'\in\Lambda\butnot\Z\rr \\ Q_1 \leq q' \leq Q_2 \\ \dist(\rr',\R\rr) \leq 2q' \Psi(q)}} \Psi^\pdim(q') \noreason\\
&\leq \int_{-(Q_2 + \cdd)}^{Q_2 + \cdd} \int_{\R^\pdim} \w\Psi^\pdim(|q'| - \cdd) \Big[\dist(\rr',\R\rr) \leq 2 k_1 |q'| \Psi(q)\Big] \;\dee\pp' \;\dee q' \noreason\\
&& \hspace{0 in} \text{(by Observation \ref{observationlatticecomparison} and Claim \ref{claimfi})}\\
&\leq \int_{-(Q_2 + \cdd)}^{Q_2 + \cdd} \w\Psi^\pdim(|q'| - \cdd) \lebesgue_{\R^\pdim}\big(B(q'\pp/q,2 k_1 k_2 |q'| \Psi(q) )\big) \;\dee q' \noreason\\
&& \hspace{-3 in} \text{(since $\dist(\rr',\R\rr) \leq k_2 \|\pp' - q'\pp/q\|$ for some constant $k_2$)}\\
&\approx_{\Irr(\rr)} \int_0^{Q_2 + \cdd} \w\Psi^\pdim(q' - \cdd) \big(q' \Psi(q)\big)^\pdim \;\dee q' \noreason\\
&= \Psi^\pdim(q) \int_{-\cdd}^{Q_2} \w\Psi^\pdim(q') (q' + \cdd)^\pdim \;\dee q' \noreason\\
&\lessapprox \Psi^\pdim(q) \int_{-2 Q_1}^{Q_2} \w\Psi^\pdim(q') (Q_1\vee q')^\pdim \;\;\dee q' \by{\eqref{diamDbound}}\\
&= \Psi^\pdim(q) [F_\psi(Q_1,Q_2) + 3Q_1 \psi^\pdim(Q_1)] \since{$\qdim = 1$}\\
&= \Psi^\pdim(\|\rr\|) [F_\psi(Q_1,Q_2) + 3\phi(Q_1)]. \by{Lemma \ref{lemmaqr}}
\end{align*}
This completes the proof of Lemma \ref{lemmaestimates2} in the case $\qdim = 1$.
\end{proof}

\begin{proof}[Completion of the proof if $\qdim\geq 2$]

\begin{sublemma}
For all $\rr' = (\pp',\qq')\in\Lambda\butnot\Z\rr$ with $\|\qq'\|\geq \|\qq\|$, we have
\begin{equation}
\label{lambdaintersection}
\lebesgue_\KK\big(\Delta_\psi(\rr)\cap\Delta_\psi(\rr')\big) \lessapprox \left(\frac{\psi(\|\rr\|) \psi(\|\rr'\|)}{\|\rr\Wedge\rr'\|}\right)^\pdim
\eucq \left(\frac{\psi(\|\rr\|)}{\|\rr\|} \frac{\psi(\|\rr'\|)}{\dist(\rr',\R\rr)}\right)^\pdim,
\end{equation}
where $\wedge$ denotes the wedge product.
\end{sublemma}
\begin{subproof}
Let $\GG(\qdim,\Dimsum)$ denote the Grassmanian space consisting of all subspaces of $\R^\Dimsum$ of dimension $\qdim$. Define the map $\iota:\KK\to \GG(\qdim,\Dimsum)$ via the formula
\[
\iota(A) = \{(A\qq,\qq) : \qq\in \R^\qdim\},
\]
and note that for all $\rr = (\pp,\qq)\in\R^\Dimsum$ and $A\in\KK$, we have $\|A\qq - \pp\| \approx \dist(\rr,\iota(A))$. Also note that $\iota_*[\lebesgue_\KK] \lessapprox \haarg$, where $\haarg$ denotes the natural measure on $\GG(\qdim,\Dimsum)$ (i.e. the unique measure on $\GG(\qdim,\Dimsum)$ invarant under the action of $\SO(\Dimsum)$).\Footnote{Further details: The Radon--Nikodym derivative $\frac{\dee \iota_*[\lebesgue_\KK]}{\dee \haarg}$ can be computed as the quotient of the two volume forms $\iota_*[\omega_\KK]$ and $\omega_\GG$ corresponding to $\iota_*[\lebesgue_\KK]$ and $\haarg$. Both of these are positive continuous volume forms, so their quotient is a positive continuous function on the compact set $\iota(\KK)$. If $C$ is an upper bound for this function, then we have $\iota_*[\lebesgue_\KK] \leq C \haarg$.} So
\[
\lebesgue_\KK\big(\Delta_\psi(\rr)\cap\Delta_\psi(\rr')\big) \lessapprox \haarg\left(\left\{V\in \GG(\qdim,\Dimsum) : \begin{array}{c}
\dist(\rr,V) \leq k_3 \psi(\|\rr\|)\\
\dist(\rr',V) \leq k_3 \psi(\|\rr'\|)
\end{array}\right\}\right)
\]
for some constant $k_3 > 0$. Here we have used Lemma \ref{lemmaqr} to replace $\psi(\|\qq\|)$ and $\psi(\|\qq'\|)$ by $\psi(\|\rr\|)$ and $\psi(\|\rr'\|)$, respectively, in the case where the left-hand side is nonzero.

Now, one way of selecting a point $V\in \GG(\qdim,\Dimsum)$ randomly with respect to $\haarg$ is to independently select $\pdim$ vectors $\ss_1,\ldots,\ss_\pdim\in \S^{\Dimsum - 1}$ randomly with respect to $\lambda_\S$ (normalized Lebesgue measure on $\S$) and then to let $V = \bigcap_{i = 1}^\pdim \ss_i^\perp$. (This can be verified by checking that the measure resulting from this method is invariant under the action of $\SO(\Dimsum)$.) In this case, $\dist(\rr,V) \geq |\rr\cdot \ss_i|$ for all $i$. It follows that the probability that $\dist(\rr,V) \leq k_3 \psi(\|\rr\|)$ and $\dist(\rr',V) \leq k_3 \psi(\rr')$ is less than the probability that for all $i$, we have $\|\rr\cdot \ss_i\| \leq k_3 \psi(\|\rr\|)$ and $\|\rr'\cdot\ss_i\| \leq k_3 \psi(\|\rr'\|)$. Translating this into symbols, we have
\begin{align*}
\haarg\left(\left\{V\in \GG(\qdim,\Dimsum) : \begin{array}{c}
\dist(\rr,V) \leq k_3 \psi(\|\rr\|)\\
\dist(\rr',V) \leq k_3 \psi(\|\rr'\|)
\end{array}\right\}\right)
&\leq \big(\lebesgue_\S(R)\big)^\pdim,
\end{align*}
where
\[
R = \left\{\ss \in \S^{\Dimsum - 1} :
|\rr\cdot\ss| \leq k_3 \psi(\|\rr\|),\;\;
|\rr'\cdot\ss| \leq k_3 \psi(\|\rr'\|)
\right\}.
\]
The region $R$ is approximately\Footnote{Here two Riemannian manifolds are ``approximately the same'' if each can be embedded into the other in a way that does not distort size too much.} the product of a $(\Dimsum - 3)$-dimensional unit sphere with a parallelogram whose distances between pairs of opposing sides are $k_3 \psi(\|\rr\|)/\|\rr\|$ and $k_3 \psi(\|\rr'\|)/\|\rr'\|$, and whose angle between sides is approximately $\|\rr\Wedge\rr'\|/(\|\rr\|\cdot\|\rr'\|)$. Computing the volume of this shape completes the proof.
%
\end{subproof}

Thus
\begin{align*}
&\frac{1}{\Psi^\pdim(\|\rr\|)} \sum_{\substack{\rr'\in\Lambda\butnot\Z\rr \\ Q_1 \leq \|\qq'\| \leq Q_2}} \lebesgue_\KK(\Delta_\psi(\rr)\cap\Delta_\psi(\rr')) \noreason\\
&\lessapprox \sum_{\substack{\rr'\in\Lambda\butnot\Z\rr \\ \|\rr'\| \leq Q_2}} \left(\frac{\w\psi(\|\rr'\|)}{\dist(\rr',\R\rr)}\right)^\pdim \by{\eqref{lambdaintersection} and Lemma \ref{lemmaqr}}\\
&\leq \int_{B(\0,Q_2 + \cdd)} \left(\frac{\w\psi(\|\rr'\| - \cdd)}{\dist(\rr',\R\rr)/k_1}\right)^\pdim \;\dee \rr' \by{Observation \ref{observationlatticecomparison} and Claim \ref{claimfi}}\\
&\lessapprox_{\Irr(\rr)} \int_0^{Q_2 + \cdd} \int_0^{Q_2 + \cdd} \left(\frac{\w\psi(x\vee z - \cdd)}{x}\right)^\pdim x^{\Dimsum - 2} \;\dee x \;\dee z &\left(
\begin{array}{c}
\text{cylindrical coordinates}\\
\text{$x = \dist(\rr',\R\rr)$, $z = \dist(\rr',\rr^\perp)$}
\end{array}
\right)\\
&= \int_0^{Q_2 + \cdd} \w\psi^\pdim(R - \cdd) \int_0^R (R^{\qdim - 2} + r^{\qdim - 2}) \;\dee r \;\dee R
\note{$R = x\vee z$, $r = x\wedge z$}\\
&\approx \int_0^{Q_2 + \cdd} \w\psi^\pdim(R - \cdd) R^{\qdim - 1} \;\dee R \since{$\qdim \geq 2$}
\end{align*}
\begin{align*}
&= \int_{-\cdd}^{Q_2} \w\psi^\pdim(R) (R + \cdd)^{\qdim - 1} \;\dee R\\
&\lessapprox \int_{-2Q_1}^{Q_2} \w\psi^\pdim(R) (Q_1 \vee R)^{\qdim - 1} \;\dee R = F_\psi(Q_1,Q_2) + 3\phi(Q_1). \by{\eqref{diamDbound}}
\end{align*}
This completes the proof of Lemma \ref{lemmaestimates2} in the case $\qdim \geq 2$.
\end{proof}

We are now ready to complete the proof of Theorem \ref{theoremestimates}:

\begin{proof}[Proof of Theorem \ref{theoremestimates}]
The $\lesssim$ direction follows immediately from \eqref{estimates11}, so we prove the $\gtrsim$ direction. Fix $K \geq 1$ satisfying $Q_1 \geq K\Irr(\Lambda)$, and consider the function $f:\KK\to\R$ defined by the formula
\[
f(A) = \sum_{\substack{\rr\in \Lrep \\ \Irr(\rr) \leq K \\ Q_1 \leq \|\qq\| \leq Q_2}} \big[A\in \Delta_\psi(\rr)\big].
\]
Then
\begin{align*}
\|f\|_2^2 - \|f\|_1
&= \int (f^2 - f) \;\dee \lebesgue_\KK\\
&= \sum_{\substack{\rr\in \Lrep \\ \Irr(\rr) \leq K \\ Q_1 \leq \|\qq\| \leq Q_2}} \sum_{\substack{\rr'\in \Lrep \\ \Irr(\rr') \leq K \\ Q_1 \leq \|\qq'\| \leq Q_2 \\ \rr' \neq \rr}} \int \big[A\in \Delta_\psi(\rr)\big]\cdot\big[A\in \Delta_\psi(\rr')\big] \;\dee \lebesgue_\KK(A) \noreason\\
&\leq 2 \sum_{\substack{\rr\in \Lrep \\ \Irr(\rr) \leq K \\ Q_1 \leq \|\qq\| \leq Q_2}} \sum_{\substack{\rr'\in \Lambda\butnot\Z\rr \\ \|\qq\| \leq \|\qq'\| \leq Q_2}} \lebesgue_\KK\big(\Delta_\psi(\rr)\cap\Delta_\psi(\rr')\big) \noreason\\
&\lessapprox_K \sum_{\substack{\rr\in \Lrep \\ \Irr(\rr) \leq K \\ Q_1 \leq \|\rr\| \leq Q_2}} \Psi^\pdim(\|\rr\|) \wbar F_\psi(Q_1,Q_2) \by{Lemma \ref{lemmaestimates2}}\\
&\leq \wbar F_\psi(Q_1,Q_2) \sum_{\substack{\rr\in \Lambda \\ Q_1 \leq \|\rr\| \leq Q_2}} \Psi^\pdim(\|\rr\|)\\
&\lessapprox \wbar F_\psi(Q_1,Q_2) \left[Q_1^\Dimsum \Psi^\pdim(Q_1) + \int_{Q_1}^{Q_2} r^{\Dimsum - 1} \Psi^\pdim(r) \;\dee r \right] \by{Corollary \ref{corollaryspherical}}\\
&= \wbar F_\psi(Q_1,Q_2)^2.
\end{align*}
Denote the implied constant by $k_4 = h(K)$, so that
\[
\|f\|_2^2 - \|f\|_1 \leq k_4 \wbar F_\psi(Q_1,Q_2)^2.
\]
Now consider the function $g:\KK\to\{0,1\}$ which is the composition of $f$ with the map $k\mapsto \big[k\geq 1\big]$. Equivalently, $g$ is the characteristic function of the set
\[
\bigcup_{\substack{\rr\in \Lrep \\ \Irr(\rr) \leq K \\ Q_1 \leq \|\qq\| \leq Q_2}} \Delta_\psi(\rr).
\]
Since $f = fg$, H\"older's inequality gives
\[
\|f\|_1 = \|fg\|_1 \leq \|f\|_2 \cdot \|g\|_2
\]
and thus
\begin{equation}
\label{g1bound}
\lebesgue_\KK\big(W_{\psi,\Lambda}(Q_1,Q_2)\big) \geq \|g\|_1 = \|g\|_2^2 \geq \frac{\|f\|_1^2}{\|f\|_2^2}
\geq \frac{\|f\|_1^2}{\|f\|_1 + k_4 \wbar F_\psi(Q_1,Q_2)^2}
= \frac{\|f\|_1}{1 + k_4 \wbar F_\psi(Q_1,Q_2)^2/\|f\|_1}\cdot
\end{equation}
It remains to work through some quantifier logic. Fix $1 < \multerror \leq 2$. By \eqref{estimates13}, there exists a constant $R_\multerror^{(1)}\geq 1$ depending only on $\multerror$ which is large enough so that
\begin{equation}
\label{Repsilondef}
\begin{cases}
K \geq R_\multerror^{(1)}\\
Q_1 \geq K\Irr(\Lambda)\\
\phi(Q_1) \leq F_\psi(Q_1,Q_2)/R_\multerror^{(1)}
\end{cases}
\Rightarrow\;\; \|f\|_1 \geq \multerror^{-1} \eta_\greeks F_\psi(Q_1,Q_2).
\end{equation}
Letting $K = R_\multerror^{(1)}$ in \eqref{Repsilondef} gives
\[
\begin{cases}
Q_1 \geq R_\multerror^{(1)} \Irr(\Lambda)\\
\phi(Q_1) \leq F_\psi(Q_1,Q_2)/R_\multerror^{(1)}
\end{cases}
\;\;\Rightarrow\;\; \|f\|_1 \geq \eta_\greeks \multerror^{-1} F_\psi(Q_1,Q_2),
\]
and since $R_\multerror^{(1)} \geq 1$ we have
\[
\phi(Q_1) \leq F_\psi(Q_1,Q_2)/R_\multerror^{(1)} \;\;\Rightarrow\;\; \wbar F_\psi(Q_1,Q_2) \leq 2 F_\psi(Q_1,Q_2).
\]
Combining with \eqref{g1bound} gives
\[
\begin{cases}
Q_1 \geq R_\multerror^{(1)} \Irr(\Lambda)\\
\phi(Q_1) \leq F_\psi(Q_1,Q_2)/R_\multerror^{(1)}
\end{cases}
\;\;\Rightarrow\;\; \|g\|_1 \geq \eta_\greeks \multerror^{-1} F_\psi(Q_1,Q_2) \frac{1}{1 + h(R_\multerror^{(1)}) (8/\eta_\greeks) F_\psi(Q_1,Q_2)};
\]
letting $R_\multerror^{(2)} = \frac{8 h(R_\multerror^{(1)})}{\eta_\greeks(\multerror - 1)}$, we have
\[
\begin{cases}
Q_1 \geq R_\multerror^{(1)} \Irr(\Lambda)\\
\phi(Q_1) \leq F_\psi(Q_1,Q_2)/R_\multerror^{(1)}\\
F_\psi(Q_1,Q_2) \leq 1/R_\multerror^{(2)}
\end{cases}
\;\;\Rightarrow\;\; 
\lebesgue_\KK\big(W_{\psi,\Lambda}(Q_1,Q_2)\big) \geq \|g\|_1 \geq \eta_\greeks \multerror^{-2} F_\psi(Q_1,Q_2).
\]
Since $\multerror$ can be chosen arbitrarily close to $1$, this completes the proof.
\end{proof}

\draftnewpage
\section{Lower bound}
\label{sectioneasydirection}

In this section, we prove the lower bound in Theorem \ref{maintheoremv2}, namely we prove that there exists a function $C_\greeks$ such that if $L_{f,\psi} > C_{\greeks}(M_\psi)$, then $\HH^f(\Bad\psi) = \infty$. As in that theorem, we let $f$ be a dimension function, and we let $\psi$ be a nice approximation function such that $L_{f,\psi} > C_{\greeks}(M_\psi)$. Let the notation be as in Section \ref{sectionpreliminaries}.

\subsection{Sketch of the proof}
Imagine a spaceship hurtling through the space $\Omega_\Dimsum$, flying in the direction of the $(g_t)$ flow. It periodically makes corrections to its course by applying an operator from the $(u_A)$ flow, where $A\in\KK$. After an infinite amount of time, the path of the spaceship will be within a bounded distance of a path of the form $(g_t u_{A_\infty} \Lambda_*)_{t\geq 0}$ for some $A_\infty\in\MM$.

Now fix $Q_0 > 0$, and suppose you want the spaceship to steer in a way such that $A_\infty\in B_{\psi,Q_0}$, but you have only limited control of its movements. Specifically, at certain times $t_1,t_2,\ldots$ you can make the spaceship apply a $u_A$ correction where you cannot choose $A$, but you can enforce some restrictions on the possibilities for $A$ as long as you allow e.g. at least $90\%$ of the possible values for $A$. With these limitations, what strategy should you use to guarantee $A_\infty\in B_{\psi,Q_0}$?

The answer is to use the mixing property of the $(g_t)$ flow to guarantee that at the times $t_1,t_2,\ldots$, the spaceship is in a bounded region whose size is independent of $\psi$. During the ``in between times'' where $t\neq t_k$ for any $k$, the spaceship can leave this bounded region but still stays in the time-dependent region $(K(\psi,t))_{t\geq 0}$ corresponding to $\psi$ according to \eqref{Kpsitdef}.

More rigorously, choose $R \geq 1$ large enough so that the set $K_R = \{\Lambda \in \Omega_\Dimsum : \Delta(\Lambda) \leq R\}$ has measure close to $1$, say $1 - \epsilon$. Due to the mixing property of the $(g_t)$ flow, if the gap sizes $t_{k + 1} - t_k$ are large enough then the set of potential directions $A$ such that applying a $u_A$ correction at time $t_k$ will cause the spaceship to be in $K_R$ at time $t_{k + 1}$ has size at least $1 - 2\epsilon$. Thus by applying appropriate restrictions to the movement corrections, we can guarantee that the spaceship is in $K_R$ at the times $t_1,t_2,\ldots$. Since the set of directions we must discard is at most size $2\epsilon$, this requirement does not significantly hamper our ability to control the spaceship. But then each time we control the ship, we start out in $K_R$, and since $K_R$ is compact ``all points look the same'' -- any calculations we do will apply with the same quantitative bounds to all points in $K_R$. But this means that the behavior of the ship at the intervals $[t_k,t_{k + 1}]$ ($k\in\N$) can be treated as independent, so then the heuristic argument of Section \ref{sectionheuristic} finishes the proof.

\subsection{The proof}
Fix $R \geq 1$ and $0 < \beta \leq 1$, and let $(N_k)_1^\infty$ be the sequence given by Lemma \ref{lemmaNk}. Fix $k\in\N$. Recall that $N^k = \prod_{j = 1}^k N_j$, and let
\begin{align*}
t_k &= \delta\log(N^k),& g^k &= g_{t_k}.
\end{align*}
This will be our sequence of times we can steer the spaceship. Note that by \eqref{deltamotivation}, $g^k u_A = u_{N^k A} g^k$ for all $A\in\MM$.

Now when we steer the spaceship at time $t_k$, we will maneuver so as to avoid the rationals represented by points in the window
\[
\QQ_k = \{(\pp,\qq)\in\Lambda_* : e^{-2R} Q^k \leq \|\qq\| < e^{-2R} Q^{k + 1} \}.
\]
We will do this under the assumption that the lattice represented by the spaceship at time $t_k$ is contained in the compact set
\[
K_R = \{\Lambda\in\Omega_\Dimsum : \Delta(\Lambda) \leq R\}.
\]
Specifically, if
\begin{align*}
\BB_k &= \{A\in \KK : g^k u_A \Lambda_* \in K_R\}\\
\BB_k' &= \BB_k \butnot\bigcup_{\rr\in \QQ_k} \Delta_\psi(\rr) = \BB_k \butnot W_\psi(e^{-2R} Q^k, e^{-2R} Q^{k + 1})
\end{align*}
then, assuming that the current trajectory of the spaceship is represented by $(g_t u_A \Lambda_*)_{t\geq 0}$ for some $A\in\BB_k$, we will steer so that the new trajectory is represented by a point of $\BB_{k + 1}'$. Fix $k_0\in\N$, and for each $k\geq k_0$ let
\[
T^k = \{\omega\in E^k : \KK_{\omega\given j} \subset \BB_j' \text{ for all $j = k_0,\ldots,k$}\}.
\]
(For $0 \leq k < k_0$, let $T^k = E^k$.) Note that:
\begin{itemize}
\item The sets $(\QQ_k)_{k_0}^\infty$ are disjoint and their union is $\{\rr\in\Lambda_* : \|\qq\| \geq e^{-2R} Q^{k_0}\}$.
\item If $T^\infty$ is the set of infinite branches through the tree $T^* = \bigcup_k T^k \subset E^*$, then $\Bad\psi\supset\pi(T^\infty)$. To see why this is the case, fix $\omega\in T^\infty$ and $\rr\in\Lambda_*$ such that $\|\qq\| \geq e^{-2R} Q^{k_0}$; there exists $k$ such that $\rr\in\QQ_k$, and
\[
\pi(\omega)\in \KK_{\omega\given k} \subset \BB_k' \subset \KK\butnot\Delta_\psi(\rr).
\]
Since $\rr$ was arbitrary, we have $\pi(\omega) \in B_{\psi,Q_0} \subset \Bad\psi$, where $Q_0 = e^{-2R} Q^{k_0}$.
\end{itemize}

\subsection{Bounding the evaporation rate of $T^*$}
Fix $\omega\in T^*$, and let $k = |\omega|$. We partition the set $E_{k + 1}\butnot T_\omega$ into two subsets:
\begin{align*}
T_\omega' &= \{a\in E_{k + 1} : \KK_{\omega a}\nsubset \BB_{k + 1}\},\\
T_\omega'' &= E_{k + 1} \butnot (T_\omega\cup T_\omega')
= \bigcup_{\rr\in\QQ_k} T_{\omega,\rr},
\end{align*}
where
\[
T_{\omega,\rr} = \{a\in E_{k + 1}\butnot T_\omega' : \KK_{\omega a} \cap \Delta_\psi(\rr) \neq \emptyset\}.
\]
\begin{claim}
\label{claimeasydirection}
We have
\begin{equation}
\label{ETST*1}
\frac{\#(T_\omega')}{(N_{k + 1})^\Dimprod} \LessLess{\substack{R\to\infty \\ M_\psi/\beta\to 0 \; (R)}} 1
\end{equation}
and
\begin{equation}
\label{ETST*2}
\frac{\#(T_\omega'')}{(N_{k + 1})^\Dimprod} \leqsim{\substack{\beta\to 0 \\ R\to \infty \\ M_\psi/\beta \to 0 \; (R)}} \eta_{\greeks} \beta.
\end{equation}
In particular, the upper evaporation rate $(P_k^+)_1^\infty$ of $T^*$ satisfies
\[
\sup_{k\geq 1} P_k^+ \leqsim{\substack{\beta\to 0 \\ R\to\infty \; (\beta) \\ M_\psi/\beta\to 0 \; (R,\beta)}} \eta_{\greeks} \beta.
\]
\end{claim}
\begin{proof}[Proof of \eqref{ETST*1}]
Fix $\epsilon > 0$. Let $\haar$ denote the unique probability Haar measure on the space of lattices $\Omega_\Dimsum$. Since $\bigcup_{R \geq 1} K_R = \Omega_\Dimsum$, there exists $R_\epsilon \geq 1$ such that
\[
\haar(K_{R_\epsilon}) > 1 - \epsilon.
\]
We observe that the group $G = \SL_\Dimsum(\R)$ can be written as the product of the two subgroups
\begin{align*}
U &\df \{u_A : A\in \MM\}\\
H &\df \{M\in G : M_{i,j} = 0 \all i = 1,\ldots,\pdim \all j = \pdim + 1,\ldots,\pdim + \qdim\}.
\end{align*}
Let $\nu_1$ be an absolutely continuous and compactly supported probability measure on $H$, and let $\nu_2$ be the image of $\lebesgue_\KK$ under the map $A\mapsto u_A$. Let $\nu$ denote the image of $\nu_1\times\nu_2$ under the map $(h,u)\mapsto hu$. Then $\nu$ is an absolutely continuous and compactly supported probability measure on $G$. It follows that the map $f:\Omega_\Dimsum \to \MM(\Omega_\Dimsum)$ defined by the formula
\[
f(\Lambda) \df \int \delta_{g\Lambda} \;\dee \nu(g) \in \MM(\Omega_\Dimsum)
\]
is continuous with respect to the norm topology on $\MM(\Omega_\Dimsum)$, the space of measures on $\Omega_\Dimsum$. (Here $\delta_x$ denotes the Dirac delta mass at $x$.) Equivalently, $f(\Lambda)$ is the measure on $\Omega_\Dimsum$ such that for all $S \subset \Omega_\Dimsum$, we have $f(\Lambda)[S] = \nu(\{g\in G : g\Lambda \in S\})$. Thus the family of functions
\[
f_t(\Lambda) \df f(\Lambda)[g_{-t} K_{R_\epsilon}] = \nu(\{g : g_t g\Lambda \in K_{R_\epsilon}\}) \;\; (t\in\R)
\]
is equicontinuous. On the other hand, by the Howe--Moore theorem \cite[Theorem 5.2]{HoweMoore}, the matrix coefficients of the natural representation of $G$ on $L^2_0(\Omega_\generic,\haar)$ vanish at infinity, meaning that the action of $G$ on $\Omega_\generic$ is mixing with respect to the natural measure $\haar$. It follows that the $(g_t)$ flow on $\Omega_\generic$ is also mixing, so for any fixed $\Lambda$ we have
\[
f_t(\Lambda) = \int \big[g_t x\in K_{R_\epsilon}\big] \frac{\dee [g\mapsto g\Lambda]_* \nu}{\dee\hspace{-0.025 in}\haar}(x) \;\dee\hspace{-0.025 in}\haar(x) \tendsto{t\to\infty \; (\Lambda)} \haar(K_{R_\epsilon}) > 1 - \epsilon.
\]
Equicontinuity implies that this convergence is uniform when $\Lambda$ is restricted to any compact set.

Since the map $h\mapsto g_t h g_{-t}$ is nonexpanding on $H$ and since $\Delta$ is uniformly continuous, there exists a constant $k_6 > 0$ such that
\begin{equation}
\label{K1def}
\Delta(g_t h u_A \Lambda) \geq \Delta(g_t u_A \Lambda) - k_6 \all A\in\KK \all h\in \Supp(\nu_1) \all \Lambda\in\Omega_\Dimsum \all t\geq 0.
\end{equation}
Also, there exists a constant $k_7 > 0$ such that
\begin{equation}
\label{K2}
|\Delta(u_{A - B}\Lambda) - \Delta(\Lambda)| \leq k_7 \all A,B\in\KK \all \Lambda\in\Omega_\Dimsum.
\end{equation}
Fix $R \geq R_\epsilon + k_6 + k_7$, and we will show that
\begin{equation}
\label{ETST*1mod}
\#(T_\omega') \leqsim{M_\psi/\beta\to 0 \; (R)} \epsilon (N_{k + 1})^\Dimprod,
\end{equation}
which suffices to prove \eqref{ETST*1}. Indeed, let $t_R$ be large enough so that $f_t(\Lambda) \geq 1 - \epsilon$ for all $\Lambda\in K_R$ and $t\geq t_R$. Equivalently,
\[
(\nu_1\times\lebesgue_\KK)\left(\left\{(h,A): \Delta(g_t h u_A \Lambda) > R_\epsilon\right\}\right) \leq \epsilon \all \Lambda\in K_R \all t\geq t_R.
\]
Applying \eqref{K1def} gives
\begin{equation}
\label{tsufflarge}
\lebesgue_\KK\left(\left\{A: \Delta(g_t u_A \Lambda) > R_\epsilon + k_6\right\}\right) \leq \epsilon \all \Lambda\in K_R \all t\geq t_R.
\end{equation}
Now fix $\omega\in T^*$, and let $k = |\omega|$. Since $\omega\in T^k$, we have $\pi(\omega) \in \BB_k'$ and thus $\Lambda_\omega \df g^k u_{\pi(\omega)} \Lambda_\ast \in K_R$. Moreover, if we assume $M_\psi/\beta \leq \delta/(\alpha t_R)$ (a valid assumption given the convergence assumption $M_\psi/\beta \to 0 \; (R)$), then \eqref{FpsiMpsibounds} and \eqref{Nkdef} together imply that $t_{k + 1} - t_k = \delta\log(N_{k + 1}) \geq t_R$. Thus, we can plug $\Lambda = \Lambda_\omega$ and $t = t_{k + 1} - t_k = \delta\log(N_{k + 1})$ into \eqref{tsufflarge}. Now by \eqref{deltamotivation}, we have
\[
g_t u_A \Lambda = g_t u_A g^k u_{\pi(\omega)} \Lambda_\ast = g^{k + 1} u_{\pi(\omega) + (N^k)^{-1} A} \Lambda_\ast = g^{k + 1} u_{\Phi_\omega(A)} \Lambda_\ast.
\]
Thus, applying a change of variables to \eqref{tsufflarge} yields
\[
(N^k)^\Dimprod \lebesgue_\KK\left(\left\{A\in \KK_\omega: \Delta(g^{k + 1} u_A \Lambda_*) > R_\epsilon + k_6 \right\}\right) \leq \epsilon.
\]
Fix $a\in E_{k + 1}$; then for all $A,B\in\KK_{\omega a}$ we have
\[
|\Delta(g^{k + 1} u_B \Lambda_*) - \Delta(g^{k + 1} u_A \Lambda_*)| = |\Delta(u_{N^{k + 1}(B - A)} g^{k + 1} u_A \Lambda_*) - \Delta(g^{k + 1} u_A \Lambda_*)| \leq k_7.
\]
It follows that
\begin{align*}
\left\{A\in \KK_\omega: \Delta(g^{k + 1} u_A \Lambda_*) > R_\epsilon + k_6\right\}
&\supset \bigcup\Big\{\KK_{\omega a} : \exists A\in \KK_{\omega a} \;\; \Delta(g^{k + 1} u_A \Lambda_*) > R_\epsilon + k_6 + k_7\Big\} \noreason\\
&\supset \bigcup_{a\in T_\omega'}\KK_{\omega a} \since{$R\geq R_\epsilon + k_6 + k_7$}
\end{align*}
and thus
\[
(N^k)^\Dimprod \lebesgue_\KK\left(\bigcup_{a\in T_\omega'}\KK_{\omega a}\right) \leq \epsilon.
\]
Direct computation and rearrangement now yields \eqref{ETST*1mod}.
\end{proof}

\begin{proof}[Proof of \eqref{ETST*2}]
Fix $\rr = (\pp,\qq)\in\QQ_k$ and $a\in T_{\omega,\rr}$. By definition, there exists $A\in \KK_{\omega a}\cap \Delta_\psi(\rr)$. Since $a\notin T_\omega'$, we have $A\in\BB_{k + 1}$ and thus
\[
g^{k + 1} u_A \rr \notin B(\0,e^{-R}).
\]
Equivalently,
\begin{equation}
\label{dichotomy}
\frac{1}{Q^{k + 1}}\|\qq\| > e^{-R} \text{ or } (N^{k + 1})^{\delta/\pdim} \|A\qq - \pp\| > e^{-R}.
\end{equation}
On the other hand, since $\rr\in\QQ_k$ we have
\[
\|\qq\| \leq e^{-2R} Q^{k + 1}.
\]
In particular, this contradicts the first half of the dichotomy \eqref{dichotomy}, so it must be the last half which holds. Thus
\[
\|A\qq - \pp\| \geq e^{-R} (N^{k + 1})^{-\delta/\pdim} = e^{-R} \frac{Q^{k + 1}}{N^{k + 1}} \geq e^R \frac{\|\qq\|}{N^{k + 1}}\cdot
\]
Since $\diam(\KK_{\omega a}) = \diamcube/N^{k + 1}$, we have for all $B \in \KK_{\omega a}$
\begin{align*}
\|B\qq - \pp\| &\leq \|A\qq - \pp\| + \|B - A\|\cdot\|\qq\|\\
&\leq \|A\qq - \pp\| + \frac{\diamcube\cdot\|\qq\|}{N^{k + 1}}\\
&\leq (1 + \diamcube e^{-R})\|A\qq - \pp\|
\leq (1 + \diamcube e^{-R})\psi(\|\qq\|).
\end{align*}
Letting $\epsilon = \epsilon_R = \diamcube e^{-R}$, we have
\[
\KK_{\omega a} \subset \Delta_{(1 + \epsilon)\psi}(\rr) \subset W_{(1 + \epsilon)\psi}(e^{-2R} Q^k,e^{-2R} Q^{k + 1}) \all \rr\in\QQ_k \all a\in T_{\omega,\rr}
\]
and thus
\[
\frac{\#(T_\omega'')}{(N^{k + 1})^\Dimprod}
= \sum_{a\in T_\omega''} \lebesgue_\KK(\KK_{\omega a})
\leq \lebesgue_\KK\left(\KK_\omega\cap W_{(1 + \epsilon)\psi}(e^{-2R} Q^k,e^{-2R} Q^{k + 1})\right).
\]
Now let $g_\omega = g^k u_{\pi(\omega)}$ and $\Lambda_\omega = g_\omega \Lambda_*$ be as in Corollary \ref{corollaryestimates}, and let $\psi'(q) = (Q^k)^{\qdim/\pdim} \psi(Q^k q)$ be as in the proof of Corollary \ref{corollaryestimates}. Fix $K\geq 1$. We have
\begin{align*}
\frac{\#(T_\omega'')}{(N_{k + 1})^\Dimprod}
&\leq \underbrace{\lebesgue_\KK\left(W_{(1 + \epsilon)\psi',\Lambda_\omega}(e^{-2R},K)\right)}_{\text{Term 1}} + \underbrace{(N^k)^\Dimprod\lebesgue_\KK\left(\KK_\omega\cap W_{(1 + \epsilon)\psi}(K Q^k,e^{-2R} Q^{k + 1})\right)}_{\text{Term 2}}.
\end{align*}
Now since $\omega\in T^k$, we have $\pi(\omega) \in \BB_k$ and thus $\Lambda_\omega\in K_R$. So $\#(\Lambda_\omega\cap \ball\0 K)$ is bounded depending only on $R$ and $K$. Thus by \eqref{crudebound},
\[
\text{Term 1} \lessapprox ((1 + \epsilon)\Psi')^\pdim(e^{-2R}) \#(\Lambda_\omega\cap \ball \0 K) \lessapprox_{R,K} \phi'(e^{-2R}) = \phi(e^{-2R} Q^k) \leq \phi(1) = M_\psi.
\]
On the other hand, we have
\begin{align*}
K/\Irr(\Lambda_\omega) &\tendsto{K\to\infty \; (R)} \infty  \since{$\Lambda_\omega\in K_R$}\\
F_{(1 + \epsilon)\psi}(Q^k,K Q^k) &\leq (1 + \epsilon)^\pdim \log(K) M_\psi \noreason\\
F_{(1 + \epsilon)\psi}(e^{-2R} Q^{k + 1},Q^{k + 1}) &\leq (1 + \epsilon)^\pdim 2 R M_\psi \noreason
\end{align*}
and thus by Corollary \ref{corollaryestimates},
\[
\text{Term 2} \eqsim{\substack{K\to \infty \; (R) \\ \beta \to 0 \\ M_\psi/\beta \to 0 \; (R,K)}} \eta_\greeks (1 + \epsilon_R) \beta \eqsim{R\to \infty} \eta_\greeks \beta.
\]
Standard quantifier logic completes the proof.
\end{proof}

%
%
%

\subsection{Completion of the proof}
For each $\multerror > 1$, we have
\[
\exp(-\multerror \eta_\greeks \beta) \leq 1 - \multerror^{1/2} \eta_\greeks \beta
\]
for all sufficiently small $\beta$, and thus by Claim \ref{claimeasydirection} there exist $h_1(\multerror),h_2(\multerror),h_3(\multerror) > 0$ (depending only on $\multerror$ and $\constants$) such that
\begin{equation}
\label{hdef}
\begin{cases}
\beta = h_1(\multerror)\\
R = h_2(\multerror)\\
M_\psi/\beta = h_3(\multerror)
\end{cases}
\;\;\Rightarrow\;\;
\left[1 - \sup_{k\geq 1} P_k^+ \geq \exp(-\multerror \eta_\greeks \beta)\right].
\end{equation}
Without loss of generality, we may assume that $h_1,h_2,h_3$ are homeomorphisms from $[1,\infty]$ to $[0,\beta_0]$, $[R_0,\infty]$, and $[0,\delta_0]$ for some $\beta_0,R_0,\delta_0 > 0$ (with $h_1,h_3$ increasing and $h_2$ decreasing). In particular, $h_1(1) = h_3(1) = 0$ while $h_2(1) = \infty$. Then the function $C_\greeks$ in the statement of Theorem \ref{maintheoremv2} can be defined by the formula
\[
C_\greeks(M) = \begin{cases}
((h_1 h_3)^{-1}(M)) \eta_\greeks & \text{if } M \leq \beta_0 \delta_0 \\
\infty & \text{if } M \geq \beta_0 \delta_0
\end{cases}.
\]
Suppose that $L_{f,\psi} > C_\greeks(M_\psi)$ (in particular this implies $C_\greeks(M_\psi) < \infty$ and thus $M_\psi < \beta_0 \delta_0$), and let $\multerror \df (h_1 h_3)^{-1}(M_\psi)$, $\beta \df h_1(\multerror)$, and $R \df h_2(\multerror)$, so that $C_\greeks(M_\psi) = \multerror \eta_\greeks$. By \eqref{hdef}, we have $1 - \sup_{k\geq 1} P_k^+ \geq \exp(-\multerror \eta_\greeks \beta)$. Thus the function $f_+$ from  \6\ref{subsectionevaporationrates} satisfies
\[
f_+(\rho) \leq \left(\prod_{j = 1}^{k(\rho)} \exp(\multerror \eta_\greeks \beta)\right) \rho^\Dimprod = \exp({\multerror \eta_\greeks \beta k(\rho)}) \rho^\Dimprod.
\]
We can relate $\beta k(\rho)$ to $\rho$ via the definition of the sequence $(N_k)_1^\infty$:
\begin{align*}
F_\psi(\rho^{-\alpha}) &\geq F_\psi\big((N^{k(\rho)})^\alpha\big) = F_\psi(Q^{k(\rho)}) \noreason\\
&= \sum_{k = 0}^{k(\rho) - 1} F_\psi(Q^k,Q^{k + 1})\\
&\geq \sum_{k = 0}^{k(\rho) - 1} \beta \by{\eqref{Nkdef}}\\
&= \beta k(\rho).
\end{align*}
Combining gives
\[
f_+(\rho) \leq \exp({\multerror \eta_\greeks F_\psi(\rho^{-\alpha})}) \rho^\Dimprod
\]
and thus by Proposition \ref{propositionevaporating}(i),
\[
\HH^f(\Bad\psi) \gtrapprox \liminf_{\rho\to 0} \frac{\rho^{-\Dimprod} f(\rho)}{\exp({\multerror \eta_\greeks F_\psi(\rho^{-\alpha})})} = \liminf_{\rho \to 0} \frac{\exp\log\left(\frac{f(\rho)}{f_\ast(\rho)}\right)}{\exp(\multerror \eta_\greeks F_\psi(\rho^{-\alpha}))}\cdot
\]
Since $L_{f,\psi} > C_\greeks(M_\psi) = \multerror \eta_\greeks$, the right-hand side is equal to $\infty$, which completes the proof.

\ignore{
Let $\tau = 1 - \sup_k P_k^+ > 0$. Then for all $k\in\N$,
\[
M_-^k/(N^k)^\Dimprod \geq \tau^k.
\]
Thus for all $1/N^{k + 1} < \rho \leq 1/N^k$, we have
\[
f_+(\rho) \leq \rho^\Dimprod/\tau^k.
\]
On the other hand,
\[
k\beta \leq F_\psi(Q^k) \leq F_\psi(\rho^{-\pdim/(\dimsum)})
\]
and so
\[
f_+(\rho) \leq \rho^\Dimprod K^{F_\psi(\rho^{-\pdim/(\dimsum)})},
\]
where $K = \tau^{-1/\beta} > 1$.

Now, since $L_{f,\psi} > C_{\greeks}(M_\psi)$, we have
\[
\log\left(\frac{f(\rho)}{f_\ast(\rho)}\right) \geq C_{\greeks}(M_\psi)F_\psi(\rho^{-\alpha})
\]
for all sufficiently small $x$. Setting $x = \rho$ and rearranging gives
\[
f(\rho) \geq \rho^\Dimprod e^{C_{\greeks}(M_\psi)F_\psi(\rho^{-\pdim/(\dimsum)})}.
\]
Letting $C_{\greeks}(M_\psi) > \log(K)$, we see that $f(\rho)/f_+(\rho)\to \infty$ as $\rho\to 0$. Thus by Proposition \ref{propositionevaporating}, we have $\HH^f(\Bad\psi) = \infty$.
}

\draftnewpage
\section{Upper bound}
\label{sectionharddirection}

In this section, we prove the upper bound in Theorem \ref{maintheoremv2}, namely we prove that there exists a function $c_\greeks$ such that if $L_{f,\psi} < c_{\greeks}(M_\psi)$, then $\HH^f(\Bad\psi) = 0$. As in that theorem, we let $f$ be a dimension function, and we let $\psi$ be a nice approximation function such that $L_{f,\psi} < c_{\greeks}(M_\psi)$. Let the notation be as in Section \ref{sectionpreliminaries}.

\subsection{Sketch of the proof}
As before, imagine a spaceship flying through $\Omega_\Dimsum$ in the direction of the $(g_t)$ flow, periodically making $(u_A)$ corrections, whose path eventually corresponds to a path of the form $(g_t u_{A_\infty} \Lambda_*)_{t\geq 0}$ for some $A_\infty\in\MM$. But this time, suppose that you want to steer in a way such that $A_\infty \in W_{\psi,Q_0}$. Instead of being able to direct the movement of the spaceship at all of the times $t_1,t_2,\ldots$ when it makes $(u_A)$ corrections, you can only restrict its movement at one of those times -- better make it count! In this scenario, what is the right strategy?

The answer again depends on the spaceship being in a bounded region, but this time, we don't have to work to move it into that region -- if the spaceship ever leaves the time-dependent region $(K(\psi,t))_{t\geq 0}$ defined by \eqref{Kpsitdef}, then by definition we will have $A_\infty\in \Approx\psi$ regardless of any choices we make. So we can assume that at time $t_k$, the spaceship is not in $K(\psi,t_k)$. From this vantage point, we can attempt to move into $W_{\psi,Q_0}$ by moving into $\Delta_\psi(\rr)$ for some $\rr\in\Lambda_*$ with $\|\qq\|\geq Q_0$. According to Corollary \ref{corollaryestimates}, this strategy will work as long as the disadvantage that comes from being in $K(\psi,t_k)$ rather than a fixed compact region is small in comparison to the total size of $W_{\psi,Q_0}$ relative to $\Lambda$, i.e. as long as
\[
F_\psi(Q^k, \Irr(\Lambda_\omega) Q^k)
\]
is small. But by \eqref{Fpsibounds},
\[
F_\psi(Q^k, \Irr(\Lambda_\omega) Q^k) \leq \phi(Q^k) \log(\Irr(\Lambda_\omega)) \lessapprox \phi(Q^k) r_\psi(t_k) \sim \phi(Q^k) |\log\phi(Q^k)| \tendsto{M_\psi \to 0} 0,
\]
where the last convergence holds because $\phi(Q^k) \leq M_\psi$.

\subsection{The proof}
Fix $0 < \beta \leq 1$, let the sequence $(N_k)_1^\infty$ (depending on $\psi$ and $\beta$) be given by Lemma \ref{lemmaNk}, and let the notation be as in Section \ref{sectionpreliminaries}. Fix $Q_0 > 0$, and for each $k$ let
\[
T^k = \{\omega\in E^k : \KK_\omega\cap B_{\psi,Q_0}\neq \emptyset\}
\]
(cf. \eqref{BpsiQdef}). Then $B_{\psi,Q_0} = \pi(T^\infty)$.

\begin{lemma}
\label{lemmaharddirection}
The lower evaporation rate $(P_k^-)_1^\infty$ of $T^*$ satisfies
\[
\liminf_{k\to\infty} P_k^- \geqsim{\substack{M_\psi^{1/2}/\beta\to 0 \\ \beta \to 0}} \eta_\greeks\beta.
\]
\end{lemma}
\begin{proof}
The idea is to choose numbers $Q_{1,k},Q_{2,k} \in [Q^k,Q^{k + 1}]$ such that $F_\psi(Q^k,Q_{1,k})$ and $F_\psi(Q_{2,k},Q^{k + 1})$ are both small in comparison to $F_\psi(Q^k,Q^{k + 1}) \geq \beta$. Then for $\epsilon > 0$ small and $\omega\in T^k$, the measure of $\KK_\omega\cap W_{(1 - \epsilon)\psi}(Q_{1,k},Q_{2,k})$ can be bounded from below by Corollary \ref{corollaryestimates}.\\

\noindent {\bf Choice of $Q_{1,k}$.}
For each $k\in\N$ let $Q_{1,k} = \exp({\phi^{-1/2}(Q^k)}) Q^k$. This choice is made so as to maximize $Q_{1,k}$ while leaving $F_\psi(Q^k,Q_{1,k})$ relatively small:
\begin{align*}
F_\psi(Q^k,Q_{1,k}) &\leq \phi(Q^k) \log(Q_{1,k}/Q^k) \by{\eqref{Fpsibounds}}\\
&= \phi(Q^k) \phi^{-1/2}(Q^k) = \phi^{1/2}(Q^k) \leq M_\psi^{1/2}.
\end{align*}
\noindent {\bf Choice of $Q_{2,k}$.}
Fix $0 < \epsilon \leq \diamcube$ small. Since $\Psi$ is continuous, for all $k$ sufficiently large there exists $Q_{2,k} \in [1,Q^{k + 1}]$ such that
\begin{equation}
\label{Q2def}
\epsilon\Psi(Q_{2,k}) = \frac{\diamcube}{N^{k + 1}} = \diamcube(Q^{k + 1})^{-1/\alpha}.
\end{equation}
This choice is motivated by the implication
\begin{equation}
\label{Q2motivation}
\KK_\omega\cap W_{(1 - \epsilon)\psi}(Q_{1,k},Q_{2,k}) \neq \emptyset \;\;\;\Rightarrow\;\;\; \KK_\omega \subset W_\psi(Q_{1,k},Q_{2,k}),
\end{equation}
which holds for all $\omega\in E^{k + 1}$ (since for such $\omega$, we have $\diam(\KK_\omega) = \diamcube/N^{k + 1}$). The following calculation shows that $F_\psi(Q_{2,k},Q^{k + 1})$ is small in comparison to $\beta$:

\begin{align*}
F_\psi(Q_{2,k},Q^{k + 1})
&\leq \phi(Q_{2,k}) \log(Q^{k + 1}/Q_{2,k}) \by{\eqref{Fpsibounds}}\\
&= \phi(Q_{2,k}) \log\left(\frac{(\diamcube/(\epsilon\Psi(Q_{2,k})))^\alpha}{Q_{2,k}}\right) \by{\eqref{Q2def}}\\
&= \phi(Q_{2,k}) [\alpha\log(\diamcube/\epsilon) + (1/\Dimsum)\log(1/\phi(Q_{2,k}))] \noreason\\
&\lessapprox_{\epsilon} \phi^{1/2}(Q_{2,k}) \since{$\phi(Q_{2,k}) \leq M_\psi \leq 1$}\\
&\leq M_\psi^{1/2}.
\end{align*}
\noindent {\bf Application of Corollary \ref{corollaryestimates}.}
The above calculations show that
\[
M_\psi/\beta, F_\psi(Q^k,Q_{1,k})/\beta, F_\psi(Q_{2,k},Q^{k + 1})/\beta \tendsto{M_\psi^{1/2}/\beta \to 0 \; (\epsilon)} 0.
\]
In particular, if $M_\psi^{1/2}/\beta$ is sufficiently small (depending on $\epsilon$), then we have $Q^k \leq Q_{1,k} \leq Q_{2,k} \leq Q^{k + 1}$ and so we can apply Corollary \ref{corollaryestimates} to get
\begin{equation}
\label{estimatesmod}
(N^k)^\Dimprod \lebesgue_\KK\big(\KK_\omega\cap W_{(1 - \epsilon)\psi}(Q_{1,k},Q_{2,k})\big) \eqsim{\substack{\exp({\phi^{-1/2}(Q^k)})/\Irr(\Lambda_\omega) \to \infty \\ \beta \to 0 \\ M_\psi^{1/2}/\beta \to 0 \; (\epsilon)}} \eta_\greeks (1 - \epsilon) \beta,
\end{equation}
where $\omega\in T^k$, $g_\omega = g^k u_{\pi(\omega)}$, and $\Lambda_\omega = g_\omega \Lambda_*$ are as in Corollary \ref{corollaryestimates}.
\begin{claim}
\label{claimtopconvergence}
The top convergence assumption can be replaced by the convergence assumption $k\to\infty \; (\psi,Q_0)$, since
\begin{equation}
\label{IrrLambdalog}
\exp({\phi^{-1/2}(Q^k)})/\Irr(\Lambda_\omega) \tendsto{\substack{M_\psi\to 0 \\ k\to\infty \; (\psi,Q_0)}} \infty.
\end{equation}
\end{claim}
\noindent{\it Proof.}
Fix $\rr\in \Lambda_*$ and let $\w\rr = g_\omega \rr$. We bound $\|\w\rr\|$ from below in four cases:
\begin{enumerate}[1.]
\item Suppose $Q_0 \leq \|\qq\| \leq Q^k$. By assumption $\KK_\omega\nsubset W_{\psi,Q_0}$, so we may choose $A\in \KK_\omega \cap B_{\psi,Q_0}$. Then
\begin{align*}
\|\w\rr\| &\approx \|g^k u_A \rr\| \geq (Q^k)^{\qdim/\pdim} \|A\qq - \pp\| > (Q^k)^{\qdim/\pdim} \psi(\|\qq\|)\\
&\geq (Q^k)^{\qdim/\pdim} \psi(Q^k) = \phi^{1/\pdim}(Q^k).
\end{align*}
\item If $\|\qq\| \geq Q^k$, then $\|\w\rr\| \geq \|\w\qq\| = (Q^k)^{-1}\|\qq\| \geq 1$.
\item Suppose $1 \leq \|\qq\| \leq Q_0$. Applying the argument of Case 1 to $Q_0 \w\rr$ shows that
\[
Q_0 \|\w\rr\| \gtrapprox (Q^k)^{\qdim/\pdim} \psi(Q_0 \|\qq\|) \geq (Q^k)^{\qdim/\pdim} \psi(Q_0^2)
\]
and thus
\[
\|\w\rr\| \gtrapprox (Q^k)^{\qdim/\pdim} Q_0^{-1} \psi(Q_0^2) \tendsto{k\to\infty \; (\psi,Q_0)} \infty.
\]
\item If $\qq = \0$ but $\pp\neq \0$, then $\|\w\rr\| = (Q^k)^{\qdim/\pdim} \|\rr\| \gtrapprox 1$.
\end{enumerate}
So if $k$ is large enough (depending on $\psi$ and $Q_0$), then
\[
\minkowski(\Lambda_\omega) \gtrapprox \phi^{1/\pdim}(Q^k)
\]
and thus
\[
\Irr(\Lambda_\omega) \lessapprox \phi^{-(2\Dimsum - 1)/\pdim}(Q^k) \LessLess{M_\psi \to 0} \exp({\phi^{-1/2}(Q^k)}).\xqedhere{4.34cm}{\triangleleft}\]
\noindent {\bf Completion of the proof of Lemma \ref{lemmaharddirection}.} By \eqref{Q2motivation}, we have
\[
W_{(1 - \epsilon)\psi,\Lambda_*}(Q_{1,k},Q_{2,k}) \subset \bigcup_{\substack{a\in E_{k + 1} \\ \KK_{\omega a} \subset W_\psi(Q_{1,k},Q_{2,k})}} \KK_{\omega a} \subset \bigcup_{\substack{a\in E_{k + 1} \\ \omega a \notin T^{k + 1}}} \KK_{\omega a}
\]
where the second inclusion holds for all $k$ sufficiently large (depending on $\psi$ and $Q_0$). Thus
\begin{eqnarray*}
&&\hspace{-.8 in} \frac{1}{(N_{k + 1})^\Dimprod} \#\{a\in E_{k + 1} : \omega a \notin T^{k + 1}\}\\
&\geqsim{k\to\infty \; (\psi,Q_0)} & (N^k)^\Dimprod \lebesgue_\KK\big(\KK_\omega\cap W_{(1 - \epsilon)\psi}(Q_{1,k},Q_{2,k})\big)\\
&\eqsim{\substack{\beta \to 0 \\ M_\psi^{1/2}/\beta \to 0 \; (\epsilon) \\ k \to \infty \; (\psi,Q_0)}}& \eta_\greeks (1 - \epsilon) \beta \;\;\;\;\;\;\;\;\;\;\;\;\;\;\; \text{(by \eqref{estimatesmod})}\\
&\eqsim{\epsilon\to 0}& \eta_\greeks \beta.
\end{eqnarray*}
Standard quantifier arguments yield
\[
\frac{1}{(N_{k + 1})^\Dimprod}\#\{a\in E_{k + 1} : \KK_{\omega a} \subset W_{\psi,Q_0}\} \geqsim{\substack{M_\psi^{1/2}/\beta\to 0 \\ \beta \to 0 \\ k\to\infty \; (\psi,Q_0)}} \eta_\greeks\beta,
\]
and taking the infimum over all $\omega\in T^k$ and then the liminf as $k\to\infty$ completes the proof of Lemma \ref{lemmaharddirection}.
\end{proof}

For each $0 < \multerror < 1$, by Lemma \ref{lemmaharddirection} there exists $\xi > 0$ (depending only on $\multerror$ and $\constants$) such that
\begin{equation}
\label{deltadef}
\begin{cases}
M_\psi^{1/2}/\beta \leq \xi\\
\beta \leq \xi
\end{cases}
\;\;\Rightarrow\;\;
\begin{cases}
\liminf_{k\to\infty} P_k^- > \multerror \eta_\greeks \beta\\
M_\psi \alpha \log(2) \leq (\multerror^{-1} - 1) \beta
\end{cases}.
\end{equation}
It follows that there exist $\xi_0 > 0$ and a decreasing homeomorphism $h:[0,1] \to [0,\xi_0]$ (in particular with $h(1) = 0$) such that for all $0 < \multerror < 1$, \eqref{deltadef} holds with $\xi = h(\multerror)$. Then the function $c_\greeks$ in the statement of Theorem \ref{maintheoremv2} can be defined by the formula
\[
c_\greeks(M) = \begin{cases}
(h^{-1}(M^{1/4}))^2 \eta_\greeks & \text{if } M \leq \xi_0^4 \\
0 & \text{if } M \geq \xi_0^4
\end{cases}.
\]
Suppose that $L_{f,\psi} < c_\greeks(M_\psi)$ (in particular this implies $c_\greeks(M_\psi) > 0$ and thus $M_\psi < \xi_0^4$), and let $\beta \df \xi \df M_\psi^{1/4}$ and $\multerror \df h^{-1}(\xi)$, so that $c_\greeks(M_\psi) = \multerror^2 \eta_\greeks$. Since $M_\psi^{1/2}/\beta = \beta = \xi$, by \eqref{deltadef} we have $\liminf_{k\to\infty} P_k^- > \multerror \eta_\greeks \beta$, say $P_k^- \geq \multerror \eta_\greeks \beta$ for all $k\geq k_0 = k_0(\psi,\beta,Q_0,\multerror)$. Then the function $f_-$ from  \6\ref{subsectionevaporationrates} satisfies
\[
f_-(\rho) \geq \left(\prod_{j = k_0}^{k(\rho)} \frac{1}{1 - \multerror \eta_\greeks \beta}\right)\rho^\Dimprod \gtrapprox_{\psi,\beta,Q_0,\multerror} \exp({\multerror \eta_\greeks \beta k(\rho)}) \rho^\Dimprod.
\]
We can relate $\beta k(\rho)$ to $\rho$ via the definition of the sequence $(N_k)_1^\infty$:
\begin{align*}
F_\psi(\rho^{-\alpha}) &\leq F_\psi\big((N^{k(\rho) + 1})^\alpha\big) = F_\psi(Q^{k(\rho) + 1}) \noreason\\
&= \sum_{k = 0}^{k(\rho)} F_\psi(Q^k,Q^{k + 1})\\
&\leq \sum_{k = 0}^{k(\rho)} [\beta + M_\psi \alpha \log(2)] \by{\eqref{Nkdef}}\\
&\leq (k(\rho) + 1) \multerror^{-1} \beta \by{\eqref{deltadef}}\\
&\approx_{\plus,\multerror,\beta} \multerror^{-1} \beta k(\rho).\noreason
\end{align*}
Combining gives
\[
f_-(\rho) \gtrapprox_{\psi,\beta,Q_0,\multerror} \exp({\multerror^2 \eta_\greeks F_\psi(\rho^{-\alpha})}) \rho^\Dimprod
\]
and thus by Proposition \ref{propositionevaporating}(ii),
\[
\HH^f(B_{\psi,Q_0}) \lessapprox_{\psi,\beta,Q_0,\multerror} \liminf_{\rho\to 0} \frac{\rho^{-\Dimprod} f(\rho)}{\exp({\multerror^2 \eta_\greeks F_\psi(\rho^{-\alpha})})} = \liminf_{\rho \to 0} \frac{\exp\log\left(\frac{f(\rho)}{f_\ast(\rho)}\right)}{\exp(\multerror^2 \eta_\greeks F_\psi(
\rho^{-\alpha}))}\cdot
\]
Since $L_{f,\psi} < c_\greeks(M_\psi) = \multerror^2 \eta_\greeks$, the right-hand side is equal to zero. Since $Q_0$ was arbitrary, we have $\HH^f(\Bad\psi) = 0$, which completes the proof.

\appendix
\draftnewpage
\section{Index of Notations}
\label{sectionnotation}

\subsection{Norms}
Throughout this paper, $\firstdim,\seconddim\in\N$ are fixed, as are norms $\|\cdot\|_\firstgreek$ and $\|\cdot\|_\secondgreek$ on $\R^\firstdim$ and $\R^\seconddim$, respectively. Whenever $\|\cdot\|$ is used without a subscript, it may indicate one of the norms $\|\cdot\|_\firstgreek$ or $\|\cdot\|_\secondgreek$, the norm
\[
\|\rr\| = \|\pp\|_\pgreek\vee\|\qq\|_\qgreek \;\;\; (\rr = (\pp,\qq)\in \R^\dimsum),
\]
or the norm
\[
\|A\| = \max_{\|\qq\|_\qgreek = 1} \|A\qq\|_\pgreek \;\;\; (A\in\MM),
\]
depending on context. If $S$ is a set, then $\|S\| = \max_{\xx\in S} \|\xx\|$, unless $S$ is a lattice in which case $\|S\| = \max_{\xx\in\R^\generic} \dist(\xx,S)$.
\subsection{Shorthand}
The following mathematical objects depend only on $\constants$:
\begin{align*}
V_\pgreek &= \text{volume of $B_\pgreek(\0,1) = \{\pp\in\R^\pdim : \|\pp\|_\pgreek \leq 1\}$} \noreason\\
\Dimsum &= \dimsum&
\Dimprod &= \dimprod\\
\delta &= \Dimprod/\Dimsum&
\alpha &= \pdim/\Dimsum\\
\eta = \eta_\greeks &= \qdim\frac{\Vgreeks}{2\zeta(\Dimsum)}&
\theta = \theta_\greeks &= \frac{\dimprod}{\dimsum}\cdot\frac{\Vgreeks}{2\zeta(\Dimsum)}\\
\psi_\ast(q) &= q^{-\qdim/\pdim}&
f_\ast(\rho) &= \rho^\dimprod\\
\Lambda_\ast &= \Z^\dimsum.
\end{align*}
The following objects also depend on $\psi$ and/or the sequence $(N_k)_1^\infty$ and/or the dimension function $f$:
\begin{align*}
\phi(q) &= (\psi(q)/\psi_\ast(q))^\pdim = q^\qdim \psi^\pdim(q),&
\Psi(q) &= \psi(q)/q\\
F_\psi(Q_1,Q_2) &= \int_{Q_1}^{Q_2} q^{\qdim - 1} \psi^\pdim(q) \;\dee q = \int_{Q_1}^{Q_2} \phi(q) \frac{\dee q}{q}&
\w\psi(q) &= \psi(Q_1\vee q)\\
\wbar F_\psi(Q_1,Q_2) &= F_\psi(Q_1,Q_2) + \phi(Q_1)&
M_\psi &= \psi^\pdim(1) = \phi(1)\\
\Delta_\psi(\rr) &= \{A\in\MM : \|A\qq - \pp\| \leq \psi(\|\qq\|)\}\noreason\\
W_{\psi,\Lambda}(Q_1,Q_2) &= \bigcup_{\substack{\rr \in \Lambda \\ Q_1 \leq \|\qq\| \leq Q_2}} \Delta_\psi(\rr)&
W_\psi(Q_1,Q_2) &= W_{\psi,\Lambda_\ast}(Q_1,Q_2)\\
W_{\psi,Q_0} &= \bigcup_{Q\geq Q_0} W_\psi(Q_0,Q)&
\Approx\psi &= \bigcap_{Q_0\in\N} W_{\psi,Q_0}\\
N^k &= \prod_{j = 1}^k N_j&
Q^k &= (N^k)^\alpha\\
t_k &= \delta\log(N^k) = \qdim\log(Q^k)&
g^k &= g_{t_k}\\
L_{f,\psi} &= \liminf_{\rho\to 0} \frac{\log\left(\frac{f(\rho)}{f_\ast(\rho)}\right)}{F_\psi(\rho^{-\alpha})}
\end{align*}

\begin{remark*}
In the special case $\psi(q) = \kappa^{1/\pdim} \psi_\ast(q)$, we get:
\begin{align*}
\phi(q) &= \kappa&
\Psi(q) &= \kappa^{1/\pdim} q^{-1/\alpha}\\
N_k &= N_* \df 2^{\lceil \beta/(\kappa\alpha\log(2))\rceil}&
N^k &= N_*^k&
Q^k &= N_*^{k\alpha}\\
M_\psi &= \phi(1) = \kappa&
L_{f,\psi} &= \frac{\dimprod - s}{\alpha\kappa}&
F_\psi(Q_1,Q_2) &= \kappa\log(Q_2/Q_1)
\end{align*}
\end{remark*}

\subsection{Miscellaneous notation organized by section of introduction}

\begin{itemize}
\item \6\ref{sectionintroduction}: $\MM$, $c_\greeks$, $C_\greeks$
\item \6\ref{sectionHD}: $\HH^f$, $\underline\BB^f$, $\boxc_\rho(S)$, $\ball x\rho$
\item \6\ref{sectionheuristic}: $\thickvar S\rho$
\item \6\ref{sectionpreliminaries}: $\approx$, $\lessapprox$, $\sim$, $\lesssim$, $\lessless$
\item \6\ref{subsectioncantorseries}: $\KK$, $\pi$, $\KK_\omega$, $\Phi_\omega$, $E_k$, $E^k$
\item \6\ref{subsectionevaporationrates}: $T_\omega$
\item \6\ref{subsectioncorrespondence}: $\Omega_\Dimsum$, $g_t$, $u_A$
\item \6\ref{sectionestimates}: $\Irr(\Lambda)$
\item \6\ref{subsectionirregularity}: $\Irr(\rr)$
\end{itemize}

\bibliographystyle{amsplain}

\bibliography{bibliography}

\end{document}